\documentclass[10pt]{article} 
\usepackage[utf8]{inputenc}
\usepackage[top=30mm, left=30mm, right=30mm, bottom=30mm]{geometry}
\usepackage{amsmath,amssymb,amsthm}
\usepackage{dsfont}
\usepackage{color}
\usepackage{lineno,hyperref}
\usepackage{yfonts}
\usepackage{appendix}
\usepackage{graphicx, subfigure}

\modulolinenumbers[5]

\newtheorem{theorem}{Theorem}[section]
\newtheorem*{theorem*}{Theorem} 
\newtheorem{lem}[theorem]{Lemma}
\newtheorem{lemma}[theorem]{Lemma}

\theoremstyle{definition}

\newtheorem{rem}[theorem]{Remark}
\newtheorem{remark}[theorem]{Remark}

\definecolor{orange}{rgb}{1.0, 0.55, 0.0}

\DeclareMathOperator*{\Fscr}{\mathcal{F}}

\begin{document}
	\title{Emergence of phase-locked states for a deterministic and stochastic Winfree model with inertia}

	\author{Myeongju Kang\footnote{Research Institute of Basic Sciences, Seoul National University, Seoul 08826, Republic of Korea. E-Mail: bear0117@snu.ac.kr},\, Marco Rehmeier\footnote{Faculty of Mathematics, Bielefeld University, 33615 Bielefeld, Germany. E-Mail: mrehmeier@math.uni-bielefeld.de }}
	
	\date{}
	\maketitle
\begin{abstract}
We study the emergence of phase-locking for Winfree oscillators under the effect of inertia. It is known that in a large coupling regime, oscillators governed by the deterministic second-order Winfree model with inertia converge to a unique equilibrium. In contrast, in this paper we show the asymptotic emergence of non-trivial synchronization in a suitably small coupling regime. Moreover, we study the effect of a new stochastically perturbed Winfree system with multiplicative noise and obtain lower estimates in probability for the pathwise emergence of such a synchronizing pattern, provided the noise is sufficiently small. We also provide numerical simulations which hint at the possibility of more general and stronger analytical results.
\end{abstract}

	
	\section{Introduction}
	Collective behavior of self-propelled particles is ubiquitously observed in natural and man-made complex systems. One particular type of such coherent dynamics which has received growing interest in the recent past is the phenomenon of \emph{synchronization}, i.e. the emergence of rhythms in interacting systems of oscillating particles, for example synchronized flashing of fireflies and firing of neurons \cite{ABPRS05, BS00, BB66, PRK01-book, S87, W79, W67}. Two mathematical models used to describe such oscillatory systems are the Kuramoto and Winfree model (\cite{K75, K84} and \cite{W67, W79} respectively). In this paper, we focus on a particular Winfree-type model, namely the \emph{second-order Winfree model with inertia} introduced in \cite{HKS21}, and our analysis includes the case of a perturbation by an external noise as well.

First, let us briefly recall the classical Winfree model as a model for systems of $N$ interacting oscillating particles, which can be visualized as rotors on the unit circle $\mathbb{S}^1$ \cite{AS01, AL14, GMNPS07, OKT17, QRS07}. Each particle of the system has its own natural frequency $\nu^i$. In the absence of interactions, the dynamics of the system are described by decoupled uniform motions of the particles $1,\dots,N$ with frequencies $\nu^1,\dots,\nu^N$, respectively. In this case, denoting the phase of particle $i$ at time $t \geq 0$ by $\theta^i_t$ (i.e. its total angular displacement from the origin as a rotor on $\mathbb{S}^1$), we have $\dot{\theta}^i_t = \nu^i$. Now each particle $i$ fires a signal  $\tilde{I}=1$ whenever it passes through the origin, i.e. $\tilde{I} = \delta_{0}$.  This signal function is usually approximated by a suitable smooth, periodic function $I$. The response $S$ from particle $j$ is a function of its current distance from the origin on $\mathbb{S}^1$. Hence the instantaneous interaction effect on particle $i$ is $S(\theta^i)\frac{1}{N}\sum_{j=1}^{N}I(\theta^j)$.  Common choices for signal and response functions are $I = 1+\cos$ and $S = -\sin$. If $\kappa \geq 0$ describes the common coupling strength within the system, one arrives at the classical first order Winfree model:
\begin{align}\label{Winfree-classical}
	\dot{\theta}^i_t = \nu^i-\frac{\kappa}{N}\sin \theta^i_t\sum_{j=1}^N\big(1+\cos\theta^j_t\big), \quad i \in \{1,\dots,N\},\, t \geq 0.
\end{align}

In contrast to the Kuramoto model, the Winfree model is not conservative. As a consequence, the mathematical analysis of \eqref{Winfree-classical} is harder, but also offers interesting asymptotic features. For example, depending on the size of $\kappa/ |\nu|$, the motions of particles in systems governed by \eqref{Winfree-classical} can vanish (\emph{oscillator death}) or can asymptotically tend to a synchronized motion (\emph{phase-locked state}).  In this paper, we focus on the latter case, but we mention the interesting results from \cite{HKM21-frustration, HKS21, HKPR16-Basic1, HPR15}, where the emergence of oscillator death (more precisely, even the existence of a unique equilibrium for $\Theta = (\theta^1,\dots,\theta^N)$) is proved for suitably confined initial data $\Theta_0$ and sufficiently large coupling coefficient $\kappa > \nu^i$. We also refer to \cite{GMNPS07} for coupling strength and natural frequency phase transition diagrams of the Winfree model.

In order to describe the emergence of nontrivial synchronization, one considers the  \textit{rotation numbers} $\rho^i, 1 \leq i \leq N$, i.e.
\begin{align*}
	\rho^i := \lim_{t \to \infty}\frac{\theta^i_t}{t},
\end{align*}
provided the limit exists. If all $\rho^i$ exist and coincide, the particles asymptotically synchronize their oscillatory dynamics into an ordered motion. More precisely, we call the case $\rho^i= 0$ for all $i$ \textit{oscillator death state} and the case $\rho^i = \rho^1 \neq 0$ for all $i$ \textit{phase-locked state}. Clearly, the latter is implied by
\begin{align}\label{def:PLS}
	\sup_{t \geq 0}\max_{1\leq i \leq N}|\theta^i_t-\theta^j_t| < \infty,
\end{align}
provided one rotation number, say $\rho^1$, exists. Hence, in this paper, we say that a phase-locked state emerges if \eqref{def:PLS} holds. In \cite{HKPR17, OKT17}, it is proven that phase-locked states emerge for sufficiently small $\kappa < |\nu|$ and under suitable restrictions on the spread $\nu^i-\nu^j$ and $\theta^i_0-\theta^j_0$ of the natural frequencies and the initial data. Further works on Winfree-type models include results on continuum limits \cite{HKM21-continuum}, adaptive couplings \cite{HKM21-adaptive}, and models with time-delay \cite{HK18-timedelay} and frustration \cite{HKM21-frustration}.

Recently, in \cite{HKS21} a new Winfree-type model, additionally taking into account the effect of inertia, has been proposed. More precisely, for a finite homogeneous inertia term $m$ and a friction coefficient $\gamma>0$, the model reads
\begin{align*}
	\begin{cases}
	\displaystyle\dot{\theta}^i_t = \omega_t^i,\\
	\displaystyle m\dot{\omega}^i _t= -\gamma \omega^i_t+\nu^i-\frac{\kappa}{N}\sum_{j=1}^N\sin \theta^i_t(1+\cos\theta^j_t),
	\end{cases}
\end{align*}
where $\Omega_t = (\omega^i_t)_{1\leq i \leq N}$ denotes the frequencies of particles $1,\dots,N$ at time $t$.
Assuming $m=1$ (otherwise consider the above system with $\tilde{\gamma},\tilde{\nu}^i$ and $\tilde{\kappa}$ instead of $\gamma, \nu^i$ and $\kappa$, where $\tilde{\gamma} := \gamma m^{-1}, \tilde{\nu}^i:= \nu^im^{-1}$ and $\tilde{\kappa}:= \kappa m^{-1}$), together with a phase-frequency initial condition $(\Theta_0,\Omega_0)$, we arrive at the Cauchy problem for the second-order Winfree model with inertia
\begin{align} \label{sys:WFI}
	\begin{cases}
		\displaystyle d\theta^i_t = \omega^i_t dt, \\
		\displaystyle d\omega^i_t  = \Big[ -\gamma\omega^i_t +\nu^i -\frac{\kappa}{N}  \sin\theta^i_t\sum_{j=1}^N(1+\cos \theta^j_t) \Big] dt, \\
		\displaystyle \big( \theta^i_t, \omega^i_t \big) \bigm|_{t=0} = \big( \theta^i_0, \omega^i_0 \big) .
	\end{cases}
\end{align}
Naturally, the emergence of asymptotic ordered behavior such as oscillator death and phase-locking are intriguing questions also for \eqref{sys:WFI}. However, the second-order nature of \eqref{sys:WFI} renders these questions more challenging compared with the classical case. A first result in this direction was obtained in \cite{HKS21}: For sufficiently large (in terms of $\nu^i$) coupling coefficient $\kappa$ and suitable initial data, the solution $(\Theta_t,\Omega_t)$ converges to a unique equilibrium (which particularly yields the emergence of oscillator death). This is coherent with the results for the first order model: If the particle interaction dominates the self-propelled individual dynamics of the particles, i.e. if $\kappa /|\nu|$ is above a threshold, then the system asymptotically tends towards oscillator death state. To the best of our knowledge, results on the emergence of phase-locked states for \eqref{sys:WFI} have not yet been obtained.

The first main result of this paper is such a phase-locking result for \eqref{sys:WFI}, and can roughly be stated as follows (see Theorem \ref{thm:det_pl} in Section \ref{sec:det} for the precise statement):
\begin{theorem*}[\textbf{Phase locking for \eqref{sys:WFI}}]
For any $\gamma>0$,  $(\nu^i)_{1\leq i \leq N}$ and sufficiently small $D>0$, if $\kappa < |\nu|$ is sufficiently small, then for all initial data $(\Theta_0,\Omega_0)$ which are sufficiently narrowly spread, phase-locking \eqref{def:PLS} emerges for the solution $(\Theta_t,\Omega_t)$ to \eqref{sys:WFI}.
\end{theorem*}

A further natural question in conjunction with complex systems of interacting particles observed in nature is the effect of external noises and its influence on the system in competition with the interaction between particles. Not only do  noise-perturbed systems often offer a more adequate description of dynamics observed in our environment, but it is also widely known that the effect of e.g. white noise can regularize an ill-posed deterministic system and thus lead to a more satisfactory mathematical analysis \cite{CG16, Flandoli-book, FGP10, Gess18, Veretennikov80}. It is hence not surprising that stochastically perturbed versions of \eqref{Winfree-classical} have been studied in the literature, both with additive \cite{K19} and multiplicative Brownian noises \cite{HKS-unpublished}, though exclusively for the emergence of oscillator death. In brief, if $\kappa >\sigma$ is suitably large, where $\sigma$ is the noise strength, the emergence of oscillator death is also observed in the stochastic case. However, for $\nu \neq 0$, the convergence is not known pathwise, but in probability \cite{HKS-unpublished}, and in \cite{K19} only estimates in probability for local in time boundedness of $\Theta_t$ are provided.

 As far as we know, a result on the emergence of phase-locked states for noisy versions of \eqref{Winfree-classical} is not known, and stochastic perturbations of the second-order model \eqref{sys:WFI} have not been considered at all in the literature.
Our second main result addresses this point: We prove a phase-locking result for the following stochastic version of \eqref{sys:WFI}, where $\omega^c := \frac 1 N \sum_{i=1}^N\omega^i$:
\begin{equation}\label{stoch-Winfree-eq}
	\begin{cases}
		\displaystyle d\theta^i_t = \omega^i_t dt, \\
		\displaystyle d\omega^i_t = \Big[-\gamma\omega^i_t + \nu^i - \kappa \sin\theta^i_t \sum_{j=1}^N \big( 1+\cos \theta^j_t \big) \Big] dt +\sigma_t(\omega^i_t-\omega^c_t)dB_t, \\
		\displaystyle (\theta^i_t,\omega^i_t)\bigm|_{t=0}=(\theta^i_0,\omega^i_0),
	\end{cases}
\end{equation}
where $B$ is a $1D$-Brownian motion and $\sigma: \mathbb{R}\to \mathbb{R}_+$ is a time-inhomogeneous noise coefficient. The interpretation of our proposed noise is the following: Each particle $i$ is subject to a common external noise, the strength of which varies with time and is proportional to the current deviation of its frequency $\omega^i_t$ from the average frequency $\omega^c_t$ of the system. A similar type of multiplicative noise was considered in \cite{AH10} for the Cucker-Smale flocking model. Our main result in the stochastic case can roughly be stated as follows (see Theorem \ref{thm-stoch-case} in Section \ref{sec:stoch} for a precise formulation):
\begin{theorem*}[\textbf{Phase locking for \eqref{stoch-Winfree-eq}}]
Let $\max_{i,j}|\nu^i-\nu^j|$ and $\kappa<|\nu|$ be sufficiently small, $\gamma$ sufficiently large and $\delta>0$ a sufficiently small number. Assume $\sigma$ is smaller in $L^\infty \cap L^2$ than some absolute constant. Then, for sufficiently narrowly spread initial data $(\Theta_0,\Omega_0)$,  phase-locking for the solution $(\Theta_t,\Omega_t)$ to \eqref{stoch-Winfree-eq} occurs pathwise with probability at least $1-2\exp\big(-\frac{\delta^2}{2||\sigma||^2}\big)$.
\end{theorem*}
Similarly as for the phase-locking results for \eqref{Winfree-classical}, we need to assume that the natural frequencies $\nu^i$ dominate the coupling strength $\kappa$. Also, for given $\delta$, the lower bound for the probability in the assertion can be made arbitrarily large in $(0,1)$, if $\sigma$ becomes sufficiently small. We remark that we did not identify a noise-induced regularizing effect on the system, but rather have to tame the noise in order to obtain phase-locking with large probability. 
\\	

This paper is organized as follows. In Section \ref{sec:det}, we study the second-order Winfree model \eqref{sys:WFI}. The main result on the emergence of phase-locking is Theorem \ref{thm:det_pl}. We also present an example of admissible choices of system parameters and initial conditions for which the result applies. In Section \ref{sec:stoch}, we introduce the stochastic model \eqref{stoch-Winfree-eq}, formulate and prove the main result in the stochastic case Theorem \ref{thm-stoch-case}. Again, we present an example of admissible parameter choices. In Section \ref{sect:numerics}, we present numerical simulations for both cases and we further provide numerical results motivating future works. Finally, Section \ref{sec:conclusion} contains a brief summary of our results.

\paragraph{Acknowledgements}
M.R. is supported by the CRC 1283 of the German Research Foundation.

\section{Second-order deterministic Winfree model with inertia}\label{sec:det}
	The following notation is used throughout the paper.
	\begin{align*}
		& \Theta_t := \big( \theta^1_t, \cdots, \theta^N_t \big), \quad \Omega_t := \big( \omega^1_t, \cdots, \omega^N_t \big), \quad \nu := \big( \nu^1, \cdots, \nu^N \big), \quad \mathcal I_c(\Theta_t) := \frac{1}{N} \sum_{i=1}^N \big( 1+\cos\theta^i_t \big), \\
		& \theta^c_t := \frac{1}{N} \sum_{i=1}^N \theta^i_t, \quad \omega^c_t := \frac{1}{N} \sum_{i=1}^N \omega^i_t, \quad \nu^c := \frac{1}{N} \sum_{i=1}^N \nu^i, \quad \theta^{ij}_t := \theta^i_t -\theta^j_t, \quad \omega^{ij}_t := \omega^i_t -\omega^j_t,
	\end{align*}
	and for ${\bf x} = (x^1, \cdots, x^N) \in \mathbb R^N$, we write $\mathcal D({\bf x}) := \max_{1\leq i, j\leq N} |x^i -x^j|$. We write $\|f\|_2$ for the $L^2$-norm (with respect to Lebesgue measure) of a measurable function $f:\mathbb{R}\to \mathbb{R}$.
	\\
	
	The aim of this section is to prove the emergence of phase-locking for the deterministic second-order $N$-particle Winfree model with inertia \eqref{sys:WFI} under suitable assumptions on the system parameters and the initial data, see Theorem \ref{thm:det_pl}.
	We start with the following lemma, which we shall use within the proof of Theorem \eqref{thm:det_pl}. For a real valued function $f$, we write $f^+ := \max(f,0)$ and $f^-:= -\min(f,0)$.
	
	\begin{lemma} \label{L4}\textup{\cite[Lemma 2.1]{OKT17}}
		 Let $x(t)$ be a solution of the differential equation
		\begin{align} \label{A-3}
			 \frac{d}{dt}x(t) = \alpha -\beta(t)x(t),\quad t \in \mathbb{R},
		\end{align}
		where $\alpha > 0$ is a constant and $\beta:\mathbb{R}\to \mathbb{R}$ is continuously differentiable and $2\pi$-periodic with
		\begin{align*}
			\int_0^{2\pi} \beta(s) ds > 0.
		\end{align*}
		Then, there exists a unique positive $2\pi$-periodic solution
		\begin{align*}
			x(t) = \frac{\displaystyle \alpha \int_0^{2\pi} \exp\bigg( -\int_\tau^{2\pi} \beta(s+t) ds \bigg) d\tau}{\displaystyle 1-\exp\bigg( -\int_0^{2\pi} \beta(s) ds \bigg)},\quad t \in \mathbb{R},
		\end{align*}
		which obeys the following bounds:
		\begin{align*}
			\frac{\displaystyle 2\alpha\pi \exp\bigg( -\int_0^{2\pi} \beta^+(s) ds \bigg)}{\displaystyle 1-\exp\bigg(-\int_0^{2\pi} \beta(s) ds \bigg)} \leq x(t) \leq \frac{\displaystyle 2\alpha\pi \exp\bigg( \int_0^{2\pi} \beta^-(s) ds \bigg)}{\displaystyle 1-\exp\bigg(-\int_0^{2\pi} \beta(s) ds \bigg)}, \quad t \in\mathbb R.
		\end{align*}
	\end{lemma}
	
	
	Moreover, we observe the following bounds for the averaged process $(\theta^c_t,\omega^c_t)$.
	\begin{lemma} \label{L1}
		Suppose the initial data and system parameters satisfy
		$$\omega^c_0 > 0 \quad \mbox{and} \quad \kappa < \frac{\nu^c}{2},$$
		and let $(\Theta_t, \Omega_t)$ be a global smooth solution of \eqref{sys:WFI}. Then, $\omega^c_t$ is uniformly bounded:
		\begin{align*}
			0 < \min\bigg\{ \omega^c_0, ~\frac{\nu^c-2\kappa}{\gamma} \bigg\} \leq \omega^c_t \leq \max \bigg\{ \omega^c_0, ~\frac{\nu^c+2\kappa}{\gamma} \bigg\}, \quad t\geq0.
		\end{align*}
		In particular, $t\mapsto \theta^c_t$ is strictly increasing and unbounded.
	\end{lemma}
	
	\begin{proof}
		Integrating the identity
		\begin{align*}
			\frac{d}{dt} \big( e^{\gamma t} \omega^i_t \big) =  e^{\gamma t} \bigg(\frac{d}{dt}\omega^i_t +\gamma\omega^i_t \bigg) = e^{\gamma t} \Big[ \nu^i -\kappa\mathcal I_c(\Theta_t) \sin\theta^i_t \Big]
		\end{align*}
		gives
		\begin{align*}
			\omega^i_t = \omega^i_0e^{-\gamma t} +e^{-\gamma t} \int_0^t e^{\gamma s} \Big[ \nu^i -\kappa\mathcal I_c(\Theta_s) \sin\theta^i_s \Big] ds.
		\end{align*}
		By summation over $i \in \{1,\dots,N\}$ and division by $N$, we obtain
		\begin{align*}
			\omega^c_t &= \omega^c_0e^{-\gamma t} +e^{-\gamma t} \int_0^t e^{\gamma s} \bigg[ \nu^c -\frac{\kappa \mathcal I_c(\Theta_s)}{N} \sum_{i=1}^N \sin\theta^i_s \bigg] ds \\
			&\geq \omega^c_0e^{-\gamma t} +e^{-\gamma t} \int_0^t e^{\gamma s} (\nu^c -2\kappa) ds \\
			&= \omega^c_0e^{-\gamma t} +\frac{\nu^c -2\kappa}{\gamma} (1-e^{-\gamma t}) \geq \min\bigg\{ \omega^c_0, ~\frac{\nu^c-2\kappa}{\gamma} \bigg\},
		\end{align*}
		Similarly, it follows
		\begin{align*}
			\omega^c_t &= \omega^c_0e^{-\gamma t} +e^{-\gamma t} \int_0^t e^{\gamma s} \bigg[ \nu^c -\frac{\kappa \mathcal I_c(\Theta_s)}{N} \sum_{i=1}^N \sin\theta^i_s \bigg] ds \\
			&\leq \omega^c_0e^{-\gamma t} +e^{-\gamma t} \int_0^t e^{\gamma s} (\nu^c +2\kappa) ds \\
			&= \omega^c_0e^{-\gamma t} +\frac{\nu^c +2\kappa}{\gamma} (1-e^{-\gamma t}) \leq \max \bigg\{ \omega^c_0, ~\frac{\nu^c+2\kappa}{\gamma} \bigg\}.
		\end{align*}
	\end{proof}

	\begin{lemma} \label{L1.2}
		Let $(\Theta_t, \Omega_t)$ be a global smooth solution of \eqref{sys:WFI}. Suppose there exist $T, D>0$ such that
		\begin{align*}
			\sup_{t\in[0, T]} \mathcal D(\Theta_t) \leq D.
		\end{align*}
		Then, we have
		\begin{align} \label{A-7}
		\sup_{t \in [0,T]}	\mathcal D(\Omega_t) \leq \max\bigg\{ \mathcal D(\Omega_0), ~\frac{\mathcal D(\nu)+2\kappa D}{\gamma} \bigg\} =: m_0 = m_0(\Omega_0,\nu,\kappa,D,\gamma).
		\end{align}
	\end{lemma}
	
	\begin{proof}
	It is
		\begin{align*}
			\frac{d}{dt} \big( e^{\gamma t} \omega^{ij}_t \big) &= e^{\gamma t} \bigg( \frac{d}{dt}\omega^{ij}_t +\gamma\omega^{ij}_t \bigg) \\
			&= e^{\gamma t} \big[ \nu^{ij} -\kappa\mathcal I_c(\Theta_t) \big(\sin\theta^i_t -\sin\theta^j_t \big) \big] \leq e^{\gamma t} (\mathcal D(\nu)+2\kappa D), \quad t\in(0, T),
		\end{align*}
	and integrating yields
		\begin{align*}
			\omega^{ij}_t \leq e^{-\gamma t} \mathcal D(\Omega_0) +\frac{\mathcal D(\nu)+2\kappa D}{\gamma} (1 -e^{-\gamma t}) \leq \max\bigg\{ \mathcal D(\Omega_0), ~\frac{\mathcal D(\nu)+2\kappa D}{\gamma} \bigg\}, \quad t\in(0, T).
		\end{align*}
	\end{proof}
	Using the previous lemma, we obtain the following refined estimates.
	\begin{lemma} \label{L3}
		Let $(\Theta_t, \Omega_t)$ be a global smooth solution of \eqref{sys:WFI}. Suppose there exist $T, D>0$ such that
		\begin{align*}
			\sup_{t\in[0, T]} \mathcal D(\Theta_t) \leq D.
		\end{align*}
		Then, for all $i, j\in\{1,\cdots,N\}$ and $t\in(0, T)$, we have
		\begin{align*}
			& \frac{d}{dt}\big( \omega^{ij}_t +\gamma\theta^{ij}_t \big) \leq \mathcal D(\nu) +\frac{2\kappa}{\gamma} m_0+3\kappa D^2 -\frac{\kappa}{\gamma} \cos\theta^c_t \big( 1+\cos\theta^c_t \big) \big( \omega^{ij}_t +\gamma \theta^{ij}_t \big), \\
			& \bigg| \frac{d}{dt}\big( \omega^{ij}_t +\gamma\theta^{ij}_t \big) \bigg| \leq \mathcal D(\nu) +\frac{4\kappa}{\gamma} m_0 +3\kappa D^2 +2\kappa D,
		\end{align*}
		where $m_0$ is defined in \eqref{A-7}.
	\end{lemma}
	
	\begin{proof}
		We differentiate $\omega^{ij}_t +\gamma\theta^{ij}_t$ to obtain
		\begin{align*}
			\frac{d}{dt}\big( \omega^{ij}_t +\gamma\theta^{ij}_t \big) &= \nu^{ij} -\kappa\mathcal I_c(\Theta_t) \big( \sin\theta^i_t -\sin\theta^j_t \big) \\
			&= \nu^{ij} -\kappa \cos\theta^c_t \big( 1+\cos\theta^c_t \big) \bigg( \frac{\omega^{ij}_t}{\gamma} +\theta^{ij}_t \bigg) +\frac{\kappa}{\gamma} \omega^{ij}_t \cos\theta^c_t \big( 1+\cos\theta^c_t \big) \\
			&\hspace{.2cm} +\kappa \theta^{ij}_t \cos\theta^c_t \big( 1+\cos\theta^c_t \big) -\kappa\mathcal I(\Theta_t) \big( \sin\theta^i_t -\sin\theta^j_t \big).
		\end{align*}
		It follows from
		\begin{align*}
			& \kappa \theta^{ij}_t \cos\theta^c_t \big( 1+\cos\theta^c_t \big) -\kappa\mathcal I(\Theta_t) \big( \sin\theta^i_t -\sin\theta^j_t \big) \\
			&\hspace{.2cm} \leq \kappa\theta^{ij}_t \cos\theta^c_t \big( 1+\cos\theta^c_t -\mathcal I(\Theta_t) \big) -\kappa\mathcal I(\Theta_t) \big( \sin\theta^i_t -\sin\theta^j_t -\theta^{ij}_t \cos\theta^c_t \big) \\
			&\hspace{.2cm} \leq \frac{\kappa D}{N} \sum_{k=1}^N \big| \cos\theta^c_t -\cos\theta^k_t \big| +2\kappa D^2 \leq 3\kappa D^2, \quad t\in(0, T),
		\end{align*}
		and Lemma \ref{L1.2} that
		\begin{align*}
			\frac{d}{dt}\big( \omega^{ij}_t +\gamma\theta^{ij}_t \big) &\leq \mathcal D(\nu) -\kappa \cos\theta^c_t \big( 1+\cos\theta^c_t \big) \bigg( \frac{\omega^{ij}_t}{\gamma} +\theta^{ij}_t \bigg) +\frac{2\kappa}{\gamma} m_0 +3\kappa D^2, \quad t\in(0, T).
		\end{align*}
		Then, the second estimate follows directly from Lemma \ref{L1.2}.
	\end{proof}
	
	By Lemma \ref{L1} and smoothness of $\theta^c_t$, $(\theta_t^c)^{-1}$ is differentiable. We set
	\begin{align*}
	R^{ij}_{t}:= \omega^{ij}_t +\gamma\theta^{ij}_t, \quad \mu(t):= \theta^c_t,\quad \tilde R^{ij}(r):= R^{ij}\circ \mu^{-1}(r),
	\end{align*}
i.e. in particular $\tilde R^{ij}(\theta^c_t)= R^{ij}_t$.
	Then, one can rephrase the previous lemma in terms of $\tilde R^{ij}$:
	
	\begin{lemma} \label{L5}
		Suppose the initial data and system parameters satisfy
		\begin{align} \label{A-4}
			0 < \nu^c-2\kappa \leq \gamma\omega^c_0 \leq \nu^c+2\kappa
		\end{align}
		and let $(\Theta_t, \Omega_t)$ be a global smooth solution of \eqref{sys:WFI}. Suppose there exist $T, D>0$ such that
		\begin{align*}
			\sup_{t\in[0, T]} \mathcal D(\Theta_t) \leq D.
		\end{align*}
	Then, for all $i, j \in \{1, \cdots, N\}$ and $r\in\big(\mu(0), \mu(T)\big)$, we have
		\begin{align} \label{A-9}
			\begin{aligned}
				& \frac{d\tilde R^{ij}}{dr}(r) \leq \alpha_D -\beta(r) \tilde R^{ij}(r), \quad \mbox{where} \\
				& \beta(r) := \frac{\kappa}{\nu^c} \cos r \big( 1+\cos r \big), \\
				& \alpha_D := \frac{\gamma \mathcal D(\nu) +2\kappa m_0 +3\gamma\kappa D^2}{\nu^c} +\frac{2\gamma^2\kappa\mathcal D(\nu) +8\gamma\kappa^2 m_0 +6\gamma^2\kappa^2 D^2 +4\gamma^2\kappa^2 D}{\nu^c(\nu^c-2\kappa)}.
			\end{aligned}
		\end{align}
		where $m_0$ is defined in \eqref{A-7}.
	\end{lemma}
	
	\begin{proof}
		By Lemma \ref{L1} and \eqref{A-4}, we have
		\begin{align*}
			\bigg| \omega^c_t -\frac{\nu^c}{\gamma} \bigg| \leq \frac{2\kappa}{\gamma}, \quad t\geq0.
		\end{align*}
		Together with Lemma \ref{L3}, this yields
		\begin{align*}
			\frac{d\tilde R^{ij}}{dr}(\mu(t)) = \frac{dR^{ij}_t}{dt}\bigg(\frac{d\mu}{dt}(t)\bigg)^{-1} \leq \frac{\gamma\mathcal D(\nu) +4\kappa m_0 +3\gamma\kappa D^2 +2\gamma\kappa D}{\nu^c-2\kappa}, \quad t\in\big(0,T\big).
		\end{align*}
		Thus, in combination with Lemma \ref{L3}, we conclude
		\begin{align*}
			& \frac{\nu^c}{\gamma}\frac{d\tilde R^{ij}}{dr} (\mu(t))= \frac{dR^{ij}_t}{dt} +\frac{d\tilde R^{ij}}{dr}(\mu(t)) \bigg( \frac{\nu^c}{\gamma} -\frac{d\mu}{dt} (t)\bigg) \\
			&\hspace{.5cm} \leq \mathcal D(\nu) +\frac{2\kappa}{\gamma} m_0 +3\kappa D^2 -\frac{\kappa}{\gamma} \cos\mu(t) \big( 1+\cos\mu(t) \big) \tilde R^{ij}(\mu(t)) \\
			&\hspace{.7cm} +\frac{2\gamma\kappa\mathcal D(\nu) +8\kappa^2 m_0 +6\gamma\kappa^2 D^2 +4\gamma\kappa^2 D}{\gamma(\nu^c-2\kappa)}, \quad t\in\big(0,T\big),
		\end{align*}
		which implies our desired result, since $\{\mu(t): t \in (0,T)\} = (\mu(0),\mu(T))$.
	\end{proof}

For the formulation of our first main result, we define the following constants.
	\begin{align}
\label{def-L-det}	L &:= \frac{\exp\big( -\int_0^{2\pi} \beta^+(s) ds \big)}{1-\exp\big(-\int_0^{2\pi} \beta(s) ds \big)} = \frac{\exp\Big( -\frac{(4+\pi)\kappa}{2\nu^c} \Big)}{1-\exp \big( -\frac{\kappa\pi}{\nu^c} \big)}, \\ \label{def-R-det}
		R &:= \frac{\exp\big( \int_0^{2\pi} \beta^-(s) ds \big)}{1-\exp\big(-\int_0^{2\pi} \beta(s) ds \big)} = \frac{\exp \Big( \frac{(4-\pi)\kappa}{2\nu^c} \Big)}{1-\exp \big( -\frac{\kappa\pi}{\nu^c} \big)}.
	\end{align}
	
	\begin{theorem} \label{thm:det_pl}
		Supposet there exists a constant $D > 0$ such that the initial data and system parameters satisfy
		\begin{align} \label{A-5}
			\begin{aligned}
				0 < \nu^c-2\kappa \leq \gamma\omega^c_0 \leq \nu^c+2\kappa, \quad 2\pi R \alpha_D < \gamma D, \quad \mathcal D(\Theta_0) < D, \quad \mathcal D(\Omega_0 +\gamma\Theta_0) \leq 2\pi L \alpha_D,
			\end{aligned}
		\end{align}
		where $\alpha_D$ is defined in \eqref{A-9}, and let $(\Theta_t, \Omega_t)$ be a global smooth solution of \eqref{sys:WFI}. Then, $\mathcal D(\Theta_t)$ is uniformly bounded:
		\begin{align*}
		\sup_{t \geq 0}	\mathcal D(\Theta_t) \leq D.
		\end{align*}
	\end{theorem}
	
	\begin{proof}
		We define the following temporal set to use a contradiction argument:
		\begin{align*}
			\mathcal T := \big\{ s\geq0: \mathcal D(\Theta_s) < D, \big\}
		\end{align*}
		Since $\big| \theta^{ij}_t \big|$ is continuous, by \eqref{A-5} $\mathcal T$ is nonempty. Suppose that
		\begin{align*}
			T := \sup\mathcal T < \infty.
		\end{align*}
		Then, for arbitrary $i, j \in \{1,\cdots,N\}$, Lemma \ref{L5} implies
		\begin{align*}
		\frac{d\tilde R^{ij}}{dr}(r) \leq \alpha_D -\beta(r) \tilde R^{ij}(r), \quad r \in (\mu(0),\mu(T)),
		\end{align*}
	and since
		\begin{align*}
			\tilde R^{ij}(\mu(0))&= R^{ij}_0 \leq \mathcal D(\Omega_0 +\gamma\Theta_0) \leq 2\pi L \alpha_D,
		\end{align*}
		comparing with Lemma \ref{L4} gives
		\begin{align*}
			\omega^{ij}_t +\gamma\theta^{ij}_t = R^{ij}_t \leq 2\pi R \alpha_D, \quad  t\in (0, T).
		\end{align*}
		Then, since $\omega^{ij}_t = \frac{d}{dt}\theta^{ij}_t$, direct calculus yields
		\begin{align*}
			\theta^{ij}_t \leq e^{-\gamma t}\theta^{ij}_0 +(1-e^{-\gamma t}) \frac{2\pi R \alpha_D}{\gamma} \leq \max\bigg\{ \mathcal D(\Theta_0), ~\frac{2\pi R \alpha_D}{\gamma} \bigg\} < D,
		\end{align*}
		and this implies, via continuity of $\theta^{ij}$, $\mathcal D(\Theta_T) < D$, which contradicts the assumed finiteness of $T$.
	\end{proof}
	\begin{remark}
	Suppose the rotation number $\rho$ of one oscillator, say $\theta^1$, exists. Then, for any $j \in \{1,\dots,N\}$, Theorem \ref{thm:det_pl} implies
			\begin{align*}
				\limsup_{t \to \infty} \bigg| \rho - \frac{\theta^j_t}{t} \bigg| \leq \limsup_{t\to\infty} \frac{|\theta^1_t -\theta^j_t|}{t}  = 0,
			\end{align*}
		i.e. the rotation number of each oscillator exists and coincides with $\rho$.
			Lemma \ref{L1} further implies
			\begin{align*}
				\rho =  \lim_{t\to\infty} \frac{\theta^c_t}{t} = \lim_{t\to\infty} \frac{1}{t} \int_0^t \omega^c_s ds \geq \frac{\nu^c-2\kappa}{\gamma} > 0.
			\end{align*}
			Therefore, in the situation of Theorem \ref{thm:det_pl}, under the additional assumption that one rotation number exists, $\Theta_t$ converges towards a complete phase-locked state.
		\end{remark}
		
		We conclude this section with an example of admissible initial data and system parameters satisfying \eqref{A-5}. To this end, fix $\gamma, \nu^c > 0$ and $D\in(0, 0.1)$, and suppose
			$$\mathcal D(\nu) = \mathcal D(\Omega_0) = 0, \quad \omega^c_0 = \frac{\nu^c}{\gamma}.$$
			Note
			\begin{align}
			\notag	\lim_{\kappa\to0+} \kappa R &= \lim_{\kappa\to0+} \frac{\kappa}{\exp \Big( \frac{(\pi-4)\kappa}{2\nu^c} \Big)-\exp \Big(  \frac{(-\pi-4)\kappa}{2\nu^c} \Big)} \\ \label{calc-R}
				&= \lim_{\kappa\to0+} \frac{1}{\frac{(\pi-4)}{2\nu^c}\exp \Big( \frac{(\pi-4)\kappa}{2\nu^c} \Big) +\frac{(\pi+4)}{2\nu^c} \exp \Big(  \frac{(-\pi-4)\kappa}{2\nu^c} \Big)} = \frac{\nu^c}{\pi}.
			\end{align}
			Hence, there is $\kappa\in(0, \nu^c/2)$ sufficiently small such that
			\begin{align*}
				\kappa R < \frac{5\nu^c}{3\pi}, \quad \kappa \bigg( \frac{20D}{3\gamma^2} +\frac{160\kappa D}{3\gamma(\nu^c-2\kappa)} +\frac{20\gamma D^2}{\nu^c-2\kappa} +\frac{40\gamma D}{3(\nu^c-2\kappa)} \bigg) \leq D-10D^2.
			\end{align*}
			Therefore, we obtain
			\begin{align*}
				\frac{2\pi R \alpha_D}{\gamma} &= \frac{2\pi \kappa R\alpha_D}{\gamma \kappa} \\
				&< \frac{10\nu^c}{3\gamma} \bigg( \frac{2\kappa D}{\gamma\nu^c} +\frac{3\gamma D^2}{\nu^c} +\frac{16\kappa^2 D}{\nu^c(\nu^c-2\kappa)} +\frac{6\gamma^2\kappa D^2}{\nu^c(\nu^c-2\kappa)} +\frac{4\gamma^2\kappa D}{\nu^c(\nu^c-2\kappa)} \bigg) \\
				&= 10 D^2 +\kappa \bigg( \frac{20D}{3\gamma^2} +\frac{160\kappa D}{3\gamma(\nu^c-2\kappa)} +\frac{20\gamma D^2}{\nu^c-2\kappa} +\frac{40\gamma D}{3(\nu^c-2\kappa)} \bigg) \leq D.
			\end{align*}
			It is easy to see that at the same time the remaining estimates of \eqref{A-5} can be satisfied as well. Indeed, it is sufficient to choose $\mathcal{D}(\Omega_0)$ sufficiently small in terms of $D, L, \alpha_D$ and $\gamma$.
			\\
			
			We did not aim to optimize the constraints and choices of the initial data and system parameters in the above example. In particular, choices $\mathcal{D}(\nu)\neq 0 \neq \mathcal{D}(\Omega_0)$ are also possible within admissible choices in for Theorem \ref{thm:det_pl}.
	
	\label{sec:stoch}
	\section{Second-order stochastic Winfree model with inertia} \label{sec:stoch}
In this section, we consider the stochastically perturbed second-order Winfree model \eqref{stoch-Winfree-eq}, in which all particles are affected by a time-dependent common noise, and its strength for particle $i$ is proportional to the deviation $\omega^i_t-\omega^c_t$ of its frequency from the instantaneous average frequency of the system. 

Let us explain the underlying probabilistic setting.
$B = (B_t)_{t \geq 0}$ is a standard real Brownian motion on a filtered probability space $(\Omega,\Fscr,(\Fscr_t)_{t \geq 0},P)$, where $(\Fscr_t)_{t \geq0}$ denotes the right-continuous and completed version of the Brownian filtration $\tilde{\Fscr}_t:= \sigma(B_s,0\leq s \leq t)$ (i.e. $(\mathcal{F}_t)_{t \geq 0}$ is the smallest filtration $\mathcal{F}_t \supseteq \tilde{\Fscr}_t$ such that all $P-$zero sets belong to $\mathcal{F}_0$ and $\mathcal{F}_t = \cap_{\varepsilon>0}\mathcal{F}_{t+\varepsilon}$). Moreover, $\sigma: \mathbb{R}\to \mathbb{R}$ is nonnegative and continuous. We do not assume $\sigma$ to be strictly positive or bounded away from $0$.
The system parameters $\gamma, \nu^i, \kappa$ are deterministic, while the initial data $(\Theta_0, \Omega_0)$ can be random. By the well-posedness theory for stochastic differential equations, it follows that \eqref{stoch-Winfree-eq} has a pathwise unique global solution on the filtered probability space fixed above (for example, writing $\theta^i_t = \int_0^t\omega^i_sds+\theta^i_0$, \eqref{stoch-Winfree-eq} can be considered a stochastic delay differential equation, which is well-posed in probabilistic strong sense).

Note that the system of equations for $(\theta^c,\omega^c)$ becomes
\begin{equation*}
	\begin{cases}
		\displaystyle d\theta^c_t = \omega^c_t dt,\\
		\displaystyle d\omega^c_t = \Big[-\gamma\omega^c_t + \nu^c - \kappa \mathcal{I}_c(\theta_t)\frac{1}{N}\sum_{i=1}^N\sin\theta^i_t\Big] dt,\\
		\displaystyle (\theta^c_t,\omega^c_t)\bigm|_{t=0}=(\theta^c_0,\omega^c_0),
	\end{cases}
\end{equation*}
i.e. the system of equations governing $(\theta^c,\omega^c)$ remains deterministic. In particular, $t\mapsto \theta^c_t$ is pathwise differentiable.
 The following auxiliary result is obtained analogously to Lemma \ref{L1}, since $\sum_{i=1}^N(\omega^c_t-\omega^i_t) = 0$.
\begin{lem} \label{L2.1stoch}
	Suppose the initial data and system parameters satisfy
	$$\omega^c_0 > 0 \quad \mbox{and} \quad \kappa < \frac{\nu^c}{2},$$
	and let $(\Theta_t, \Omega_t)$ be a global solution of \eqref{stoch-Winfree-eq}. Then, $\omega^c_t$ is uniformly bounded:
	\begin{align*}
		0<\min\bigg\{ \omega^c_0, ~\frac{\nu^c-2\kappa}{\gamma} \bigg\} \leq \omega^c_t \leq \max \bigg\{ \omega^c_0, ~\frac{\nu^c+2\kappa}{\gamma} \bigg\}, \quad t\geq0.
	\end{align*}
	In particular, $t\mapsto \theta^c_t$ is strictly increasing and unbounded.
\end{lem}
We prove the emergence of phase-locking for particles $\theta^1,\dots,\theta^N$ governed by \eqref{stoch-Winfree-eq} under suitable assumptions on $\sigma$ and the system parameters $\gamma, \nu,\kappa$.
Suppose
\begin{equation}\label{sigma-assumption}
||\sigma||_2 <\infty \text{ and }\|\sigma\|_\infty \leq \sqrt{\frac{4\kappa}{\gamma}}.
\end{equation}
We need the following notation (compare with the corresponding constants from Section \ref{sec:det}). For $D,\delta >0$, set
\begin{align}\label{YDOmega-def}
	c_0 := c_0(\gamma, \kappa, \nu, \Omega_0,D,\delta) :=\frac{1}{\cosh\delta} \max\Bigg\{ \mathcal D(\Omega_0), ~\frac{\big(\mathcal D(\nu) +2\kappa D\big)\exp\Big( \frac{\|\sigma\|_2^2}{2} +\delta \Big)}{\gamma } \Bigg\} 
\end{align}
and
\begin{align}\label{alpha-def}
	&	\alpha_D := \frac{\gamma}{\nu^c} \bigg( \gamma c_0e^\delta \sinh\delta  +\frac{e^{\delta} \mathcal D(\nu)}{\cosh\delta} +3\kappa D^2 +2\kappa D \tanh\delta +\frac{9\kappa}{4\gamma} c_0 \bigg) \\\notag
	&\hspace{1.5cm} +\frac{2\gamma\kappa}{\nu^c(\nu^c-2\kappa)} \bigg( \gamma c_0e^\delta \sinh\delta  +\frac{e^{\delta} \mathcal D(\nu)}{\cosh\delta} +3\kappa D^2 +\frac{2\kappa D e^\delta}{\cosh\delta} +\frac{17\kappa}{4\gamma} c_0 \bigg).
\end{align}
Also, we use $\beta$, $L$ and $R$ as defined in  \eqref{A-9}-\eqref{def-R-det}.
Our main result is the following 
\begin{theorem}\label{thm-stoch-case}
	For $D>0$, let $(\Theta,\Omega)$ be a global solution of \eqref{A-7} with $\mathcal{D}(\Theta_0) < D$ and $0 < \nu^c-2\kappa \leq \gamma\omega^c_0 \leq \nu^c+2\kappa $. Further assume \eqref{sigma-assumption}, and let $\delta>0$ such that
	\begin{align}\label{eq-15}
		\frac{\omega_0^{ij}}{\cosh \delta}+\gamma \theta^{ij}_0 \leq2\pi L \alpha_D,
	\end{align}
	and
	\begin{align}\label{eq-14}
		~\frac{1}{\gamma} \big( 2\pi R \alpha_D+c_0e^\delta \sinh\delta  \big) < D,
	\end{align}
	where $c_0, \alpha_D$ are defined in \eqref{YDOmega-def}-\eqref{alpha-def}, and $\beta, L$ and $R$ are defined in \eqref{A-9}-\eqref{def-R-det}.
	Then, there is a measurable set $A_\delta$ with 
	\begin{align}\label{eq-13}
		P(A_\delta) \geq 1-2\exp\bigg( -\frac{\delta^2}{2\|\sigma\|_2^2} \bigg),
	\end{align}
	on which $\sup_{t \geq 0}\mathcal{D}(\Theta_t) < D$ holds pathwise.
\end{theorem}
\begin{rem}
	It is clear that all assumptions of Theorem \ref{thm-stoch-case} remain valid, if $||\sigma||_2$ becomes smaller while $\delta>0$ is fixed. Hence, for any $\varepsilon>0$, one can choose $||\sigma||_2$ sufficiently small in order to obtain $P(A_\delta) >1-\varepsilon$.
\end{rem}
For the proof, we shall use the process $Y_t = Y_0\exp\bigg( -\int_0^t \sigma_s dB_s \bigg)$, $Y_0 >0$, which is a martingale (with respect to the natural Brownian filtration). It turns out helpful to choose the (deterministic) initial condition $Y_0 = \frac{1}{\cosh \delta}$, where $\delta >0$ is as in the assertion of Theorem \ref{thm-stoch-case}.
\begin{lem}
	The process $Y$ solves the stochastic differential equation
	\begin{equation}\label{SDE-Y}
		dY_t = \frac{\sigma_t^2}{2} Y_t dt -\sigma_t Y_t dB_t,\quad t \geq 0
	\end{equation}
	(in strong probabilistic sense, i.e. on the specified underlying probability space $(\Omega,\Fscr, (\mathcal{F}_t)_{t \geq 0}, P)$). Moreover, for $\delta>0$, we have
	\begin{align*}
		P\bigg\{ \sup_{t\geq0} \bigg| \int_0^t \sigma_s dB_s \bigg| < \delta \bigg\} \geq 1-2\exp\bigg( -\frac{\delta^2}{2\|\sigma\|_2^2} \bigg).
	\end{align*}
	Consequently, for $A_\delta := \bigg\{ \sup_{t\geq0} \bigg| \int_0^t \sigma_s dB_s \bigg| < \delta\bigg\}$ we have
	$$P(A_\delta) \nearrow 1\text{ as }\delta \nearrow \infty.$$
	Moreover, setting $Y_0 := \frac{1}{\cosh \delta}$, on $A_\delta$ we have the estimates
	\begin{align}\label{eq-Y-1}
		\frac{e^{-\delta}}{\cosh\delta} = e^{-\delta} Y_0 \leq Y_t \leq e^\delta Y_0 = \frac{e^{\delta}}{\cosh\delta},\quad t \geq 0
	\end{align}
	and
	\begin{align}\label{eq-0}
		|Y_t-1| < \tanh\delta,\quad t \geq 0.
	\end{align}
\end{lem}
\begin{proof}
	It follows from Itô's formula that $Y$ solves \eqref{SDE-Y}. The estimate for $P(A_\delta)$ follows from \cite[Lemma B.1.3]{BG06-book}. The final estimates for $Y_t$ on $A_\delta$ are verified by elementary calculations.
\end{proof}
We proceed to the proof of Theorem \ref{thm-stoch-case}.
\begin{proof}[Proof of Theorem \ref{thm-stoch-case}]
	
	Note that \eqref{stoch-Winfree-eq} gives
	\begin{align*}
		d\omega^{ij}_t  = \Big[ -\gamma\omega^{ij}_t +\nu^{ij} -\kappa\mathcal I_c(\Theta_t) \big( \sin\theta^i_t -\sin\theta^j_t \big) \Big] dt +\sigma_t \omega^{ij}_t dB_t, \quad t \geq 0.
	\end{align*}
	Hence, Itô's product rule implies
	\begin{align}\label{eq-3}
		d \big( Y_t\omega^{ij}_t \big) = Y_t\Big[ -\gamma\omega^{ij}_t +\nu^{ij} -\kappa\mathcal I_c(\Theta_t) \big( \sin\theta^i_t -\sin\theta^j_t \big) \Big] dt -\frac{\sigma_t^2}{2} Y_t \omega^{ij}_t dt,\quad t \geq 0.
	\end{align}
	In particular, $t \mapsto Y_t\omega^{ij}_t$ is pathwise differentiable.
	For $D>0$, we denote by $T$ the map
	\begin{align*}
		T := \inf \{ t>0: \mathcal D(\Theta_t) > D\},
	\end{align*}
	and note that due to the continuity of $t \mapsto \theta^i_t$, $\mathcal{D}(\Theta_0) < D$ implies $T>0$ and $\mathcal{D}(\Theta_T) =D$, if $T <\infty$. In order to prove Theorem \ref{thm-stoch-case}, we show
	\begin{equation}\label{stopp-time-infinity}
		T=\infty \text{ on }A_\delta,
	\end{equation}
	which gives $\sup_{t \geq 0}\mathcal{D}(\Theta_t) \leq D$ on $A_\delta$ and hence the assertion.
	
	First note the following estimate on $A_\delta$ for $0<t <T$, which follows from $|\mathcal I_c(\Theta_t)|\leq 2$ and $|\sin x -\sin y|\leq |x-y|$:
	\begin{align*}
		& \frac{d}{dt} \bigg[ Y_t\omega^{ij}_t \exp\bigg( \gamma t +\int_0^t \frac{\sigma_s^2}{2} ds \bigg) \bigg] \\
		&\hspace{.2cm} = Y_t\Big[ \nu^{ij} -\kappa\mathcal I_c(\Theta_t) \big( \sin\theta^i_t -\sin\theta^j_t \big) \Big] \exp\bigg( \gamma t +\int_0^t \frac{\sigma_s^2}{2} ds \bigg) \\
		&\hspace{.2cm} \leq \frac{\mathcal D(\nu) +2\kappa D}{\cosh\delta} \exp\bigg( \gamma t +\frac{\|\sigma\|_2^2}{2} +\delta \bigg),
	\end{align*}
	and thus we also have
	\begin{align*}
		Y_t\omega^{ij}_t e^{\gamma t} \leq Y_t\omega^{ij}_t \exp\bigg( \gamma t +\int_0^t \frac{\sigma_s^2}{2} ds \bigg) \leq \frac{\omega^{ij}_0}{\cosh\delta} +\frac{\mathcal D(\nu) +2\kappa D}{\cosh\delta} \frac{e^{\frac{\|\sigma\|_2^2}{2} +\delta}}{\gamma} (e^{\gamma t}-1),
	\end{align*}
	from which we infer
	\begin{align}
		Y_t\omega^{ij}_t &\leq \frac{\mathcal D(\Omega_0)}{\cosh\delta} e^{-\gamma t} +\frac{\mathcal D(\nu) +2\kappa D}{\cosh\delta} \frac{e^{\frac{\|\sigma\|_2^2}{2} +\delta}}{\gamma} (1-e^{-\gamma t}) \notag\\
		&\leq \frac{1}{\cosh\delta} \max\Bigg\{ \mathcal D(\Omega_0), ~\frac{\mathcal D(\nu) +2\kappa D}{\gamma \exp\Big( -\frac{\|\sigma\|_2^2}{2} -\delta \Big)} \Bigg\} =c_0.\label{eq-4}
	\end{align}

	Combining with \eqref{eq-Y-1}, we have
	\begin{align}\label{eq-00}
		\omega^{ij}_t \leq c_0e^\delta \cosh\delta 
	\end{align}
	on $A_\delta$ and for $t <T$.  From \eqref{eq-0},\eqref{eq-3}, \eqref{eq-4} and \eqref{eq-00}, we infer on $A_\delta$ for $t <T$ 
	\begin{align}\label{eq-8}
		\frac{d}{dt} \big( Y_t\omega^{ij}_t +\gamma\theta^{ij}_t \big) &= \gamma (1-Y_t) \omega^{ij}_t +Y_t \nu^{ij} -\frac{\kappa}{\gamma} (1+\cos\theta^c_t)\cos\theta^c_t \big( Y_t\omega^{ij}_t +\gamma\theta^{ij}_t \big) \notag\\
		&\hspace{.2cm} +\kappa(1+\cos\theta^c_t)\cos\theta^c_t \theta^{ij}_t -\kappa \mathcal I_c(\Theta_t) \big( \sin\theta^i_t -\sin\theta^j_t \big) \notag\\\
		&\hspace{.2cm} -\kappa (Y_t-1)\mathcal I_c(\Theta_t) \big( \sin\theta^i_t -\sin\theta^j_t \big) +\bigg( \frac{\kappa}{\gamma} (1+\cos\theta^c_t)\cos\theta^c_t -\frac{\sigma_t^2}{2} \bigg) Y_t \omega^{ij}_t \\\
		&\leq \gamma c_0e^\delta \sinh\delta +\frac{e^{\delta} \mathcal D(\nu)}{\cosh\delta} -\frac{\kappa}{\gamma} (1+\cos\theta^c_t)\cos\theta^c_t \big( Y_t\omega^{ij}_t +\gamma\theta^{ij}_t \big) \notag\\\
		&\hspace{.2cm} +3\kappa D^2 +2\kappa D \tanh\delta +\frac{9\kappa}{4\gamma} c_0\notag,
	\end{align}
	and consequently
	\begin{align}\label{eq-10}
		\bigg| \frac{d}{dt} \big( Y_t\omega^{ij}_t +\gamma\theta^{ij}_t \big) \bigg| \leq \gamma c_0e^\delta \sinh\delta  +\frac{e^{\delta} \mathcal D(\nu)}{\cosh\delta} +3\kappa D^2 +\frac{2\kappa D e^\delta}{\cosh\delta} +\frac{17\kappa}{4\gamma}c_0.
	\end{align}
	For abbreviation, we set 
	$$Q^{ij}_t := Y_t{\omega^{ij}_t}+\gamma \theta^{ij}_t, \,\, \mu(t):= \theta^c_t,\,\,\tilde{Q}^{ij}(r) := Q^{ij}\circ \mu^{-1}(r),$$
	i.e. in particular $\tilde{Q}^{ij}(\theta^c_t) = Q^{ij}_t$. Since $\frac{d\mu}{dt} (t) = \frac{d\theta^c}{dt}(t) = \omega^c_t $,  by Lemma \ref{L2.1stoch} and \eqref{eq-10} we obtain 
	\begin{align*}
		\frac{d\tilde Q^{ij}}{dr}(\mu(t)) = \frac{dQ^{ij}_t}{dt}\big(\frac{d\mu}{dt}(t)\big)^{-1} &\leq \frac{\gamma}{\nu^c-2\kappa} \bigg( \gamma c_0e^\delta \sinh\delta  +\frac{e^{\delta} \mathcal D(\nu)}{\cosh\delta} +3\kappa D^2 +\frac{2\kappa D e^\delta}{\cosh\delta} +\frac{17\kappa}{4\gamma} c_0 \bigg).
	\end{align*}
	Since this estimate holds for all $1\leq i,j \leq N$, combined with \eqref{eq-8} and Lemma \ref{L2.1stoch}, it implies
	\begin{align*}
		\frac{\nu^c}{\gamma} \frac{d\tilde Q^{ij}}{dr} (\mu(t))&= \frac{dQ^{ij}_t}{dt} +\frac{d\tilde Q^{ij}}{dr}(\mu(t)) \bigg( \frac{\nu^c}{\gamma} -\frac{d\mu}{dt} (t)\bigg) \\
		&\leq -\frac{\kappa}{\gamma}(1+\cos\theta^c_t)\cos\theta^c_t \big( Y_t \omega^{ij}_t +\gamma\theta^{ij}_t \big) \\
		&\hspace{.2cm} +\gamma c_0e^\delta \sinh\delta  +\frac{e^{\delta} \mathcal D(\nu)}{\cosh\delta} +3\kappa D^2 +2\kappa D \tanh\delta +\frac{9\kappa}{4\gamma} c_0 \\
		&\hspace{.2cm} +\frac{2\kappa}{\nu^c-2\kappa} \bigg( \gamma c_0e^\delta \sinh\delta  +\frac{e^{\delta} \mathcal D(\nu)}{\cosh\delta} +3\kappa D^2 +\frac{2\kappa D e^\delta}{\cosh\delta} +\frac{17\kappa}{4\gamma} c_0 \bigg),
	\end{align*}
	and therefore we obtain
	\begin{equation*}
		\frac{d\tilde Q^{ij}}{dr}(r) \leq \alpha_D -\beta(r) \tilde Q^{ij}(r), \quad r \in (\theta^c_0,\theta^c_T).
	\end{equation*}
	Comparing with Lemma \ref{L4} and since 
	$$\tilde{Q}^{ij}(\theta^c_0)  = Q^{ij}(0)= \frac{\omega_0^{ij}}{\cosh \delta}+\gamma \theta^{ij}_0 \leq2\pi L \alpha_D, $$ we have $\tilde{Q}^{ij}(r)\leq 2\pi R \alpha_D$ for all $r \in (\theta^c_0,\theta^c_T)$, so for $t \in (0,T)$, we have $Q^{ij}_t = \tilde{Q}^{ij}(\theta^c_t)\leq 2\pi R \alpha_D$.  From the definition of $Q^{ij}$, together with \eqref{eq-0} and \eqref{eq-00}, this implies the following: Under the assumptions of the assertions, on $A_\delta$ we have 
	\begin{align*}
		\omega^{ij}_t +\gamma\theta^{ij}_t = Q^{ij}_t +(1-Y_t)\omega^{ij}_t \leq 2\pi R\alpha_D+c_0e^\delta \sinh\delta ,\quad t \in (0,T).
	\end{align*}
	To conclude the proof (i.e. in order to prove \eqref{stopp-time-infinity}), we claim for each $1\leq i,j \leq N$ 
	\begin{align}\label{eq-11}
		\theta^{ij}_t \leq \max\bigg\{ \mathcal D(\Theta_0), ~\frac{1}{\gamma} \big( 2\pi R \alpha (D) +c_0e^\delta \sinh\delta  \big) \bigg\} < D,\quad t \in (0,T).
	\end{align}
	on $A_\delta$. For any path $t \mapsto \theta^{ij}_t$ with \eqref{eq-11}, assuming $T<\infty$ leads to a contradiction, since in this case $\theta^{ij}_t \longrightarrow D$ would hold as $t \to T$.
	Finally, \eqref{eq-11} can be obtained as follows: Since $\frac{d}{dt}(e^{\gamma t}\theta^{ij}_t) = e^{\gamma t}(\omega^{ij}_t +\gamma\theta^{ij}_t )$, we have
	\begin{align*}
		\theta^{ij}_t \leq \theta^{ij}_0 e^{-\gamma t} +\frac{1}{\gamma} (2\pi R\alpha_D+c_0e^\delta \sinh\delta  )(1-e^{-\gamma t}) \leq \max\bigg\{ \mathcal{D}(\Theta_0), ~\frac{1}{\gamma} \big( 2\pi R\alpha_D+c_0e^\delta \sinh\delta \big) \bigg\}.
	\end{align*}
	Since the final strict inequality in \eqref{eq-11} holds by assumption, the proof is complete.
\end{proof}
\begin{rem}
	The theorem remains valid in the case where $(\Theta_0,\Omega_0)$ is random such that for this initial data \eqref{stoch-Winfree-eq} has a unique probabilistic strong solution, if one replaces $A_\delta$ in the assertion by $\mathcal{A}_\delta:= A_\delta \cap \{\nu^c+2\kappa \geq \gamma\omega^c_0 \geq \nu^c-2\kappa\}\cap \{\mathcal{D}(\Theta_0)< D\}$. However, if the latter two sets are not of full $P$-measure, then the lower bound \eqref{eq-13} does not necessarily hold with $\mathcal{A}_\delta$ in place of $A_\delta$.
\end{rem}

We conclude this section with an example of system parameters and initial data which satisfy all assumptions of Theorem \ref{thm-stoch-case}. Suppose 
\begin{align}\label{eq00}
	\mathcal D(\nu) = 0, \quad D\in(0, 0.05),\quad\nu^c > 0.
\end{align}
 It follows from 
$
	\lim_{\kappa\to0+} \kappa R = \frac{\nu^c}{\pi}
$
that we can choose $\kappa\in(0, \nu^c/2)$ sufficiently small  such that 
\begin{align} \label{cond1}
	\kappa R \leq \frac{2\nu^c}{\pi}, \quad \frac{8\kappa(3D^2 +4D)}{\nu^c-2\kappa} < \frac{D-20D^2}{2}.
\end{align}
Let $\gamma>0$ be sufficiently large so that
\begin{align} \label{cond2}
	\frac{16\kappa D}{\gamma} +\frac{36\kappa D}{\gamma^2} +\frac{32\kappa^2D}{\gamma(\nu^c-2\kappa)} +\frac{136\kappa^2D}{\gamma^2(\nu^c-2\kappa)} +\frac{4\kappa^2D}{\gamma^3} < \frac{D-20D^2}{2},
\end{align}
and choose $\delta > 0$ sufficiently small such that
\begin{align} \label{cond3}
	\gamma e^\delta \sinh\delta \leq \kappa, \quad \tanh\delta \leq D, \quad \frac{e^\delta}{\cosh\delta} < \sqrt{2}.
\end{align}
If $\mathcal D(\Omega_0)$ is sufficiently small so that $\gamma \mathcal D(\Omega_0) \leq 2\kappa D$, then, independently from the choice of $\sigma$, we have
\begin{align*}
c_0= \frac{2\kappa D\exp\Big( \frac{\|\sigma\|_2^2}{2} +\delta \Big)}{\gamma \cosh\delta } \leq \frac{2\sqrt{2}\kappa D\exp\Big( \frac{\|\sigma\|_2^2}{2} \Big)}{\gamma }.
\end{align*}
Hence, also  choosing $\sigma$ such that $e^{\|\sigma\|_2^2/2} \leq \sqrt{2}$, we have
\begin{align}\label{eq-100}
	c_0 \leq 4\kappa D/\gamma.
\end{align}
Combining \eqref{cond1}-\eqref{eq-100}, we obtain
\begin{align*}
 &\frac{2\pi R \alpha_D +e^\delta \sinh\delta c_0}{\gamma} 
	\leq \frac{4\nu^c \alpha_D}{\gamma\kappa} +\frac{4\kappa^2D}{\gamma^3}\\
	& \leq \frac{4}{\kappa} \bigg( \kappa c_0 +5\kappa D^2 +\frac{9\kappa}{4\gamma} c_0 \bigg) +\frac{8}{\nu^c-2\kappa} \bigg( \kappa c_0 +3\kappa D^2 +4\kappa D +\frac{17\kappa}{4\gamma} c_0 \bigg) +\frac{4\kappa^2D}{\gamma^3}\\
	& \leq 20D^2 +\frac{16\kappa D}{\gamma} +\frac{36\kappa D}{\gamma^2} +\frac{32\kappa^2D}{\gamma(\nu^c-2\kappa)} +\frac{136\kappa^2D}{\gamma^2(\nu^c-2\kappa)} +\frac{4\kappa^2D}{\gamma^3} +\frac{8\kappa(3D^2 +4D)}{\nu^c-2\kappa} \\
	&<20D^2 +D -20D^2 = D.
\end{align*}
Hence, for the choices made in  \eqref{eq00}, to obtain \eqref{eq-14} one can choose $\kappa = \kappa(\nu^c)$ sufficiently small, $\gamma = \gamma(\kappa,D)$ sufficiently large, $\delta=\delta(\gamma, D)$ and $\mathcal{D}(\Omega_0) = \mathcal{D}(\Omega_0)(\gamma, \kappa, D)$ sufficiently small and, finally, $||\sigma||_2$ smaller than an absolute constant. It is obvious that these choices can be made such that also \eqref{eq-15} and $0 < \nu^c-2\kappa \leq \gamma\omega^c_0 \leq \nu^c+2\kappa $ hold.

We point out that we did not aim to optimize the constraints on the system parameters and the initial data in this example. In particular, it is not necessary to have $\mathcal{D}(\nu) =0$.

	\section{Numerical simulations}\label{sect:numerics}
Here we provide several numerical examples in order to confirm our results from Sections \ref{sec:det} and \ref{sec:stoch} and to motivate possible future works. In all simulations, we set the number of oscillators $N = 21$, the time step size $\Delta t = 0.01$, and we used the Euler method and the Euler-Maruyama method for the deterministic and stochastic case, respectively.
	
	\subsection{Deterministic case}
	We observe an example of Theorem \ref{thm:det_pl} and present further motivating examples. In the first simulation, we choose natural frequencies $\nu_i$, coupling strength $\kappa$, friction coefficient $\gamma$, and initial frequency $\Omega_0$ as follows:
	\begin{align*}
		& \nu_i = 128+10^{-4}(i-11) \implies \nu_c = 128, \quad \mathcal D(\nu) = 2\times 10^{-3}, \\
		& \kappa = 0.2, \quad \gamma = 4, \quad \omega^c_0 = \frac{\nu^c}{\gamma}, \quad \mathcal D(\Omega_0) = 0 \implies 0 < \nu^c-2\kappa \leq \gamma\omega^c_0 \leq \nu^c+2\kappa.
	\end{align*}
	For the initial condition, we choose
	\begin{align*}
		\theta^i_0 = 4\times 10^{-3}(i-11) \implies \mathcal D(\Theta_0) = 0.08,
	\end{align*}
	so the assumptions of Theorem \ref{thm:det_pl} hold with $D = 0.1$:
	\begin{align*}
		2\pi R \alpha_D -\gamma D \approx -0.0117 < 0, \quad \mathcal D(\Omega_0 +\gamma\Theta_0) -2\pi L \alpha_D \approx -0.0659 \leq 0.
	\end{align*}
	
	\begin{figure}[h!]
		\centering
		\mbox{ \hspace{-1cm}
			\subfigure[~Graph of $\mathcal D(\Theta(t))$ for $0\leq t\leq 5$]{\includegraphics[scale = 0.14]{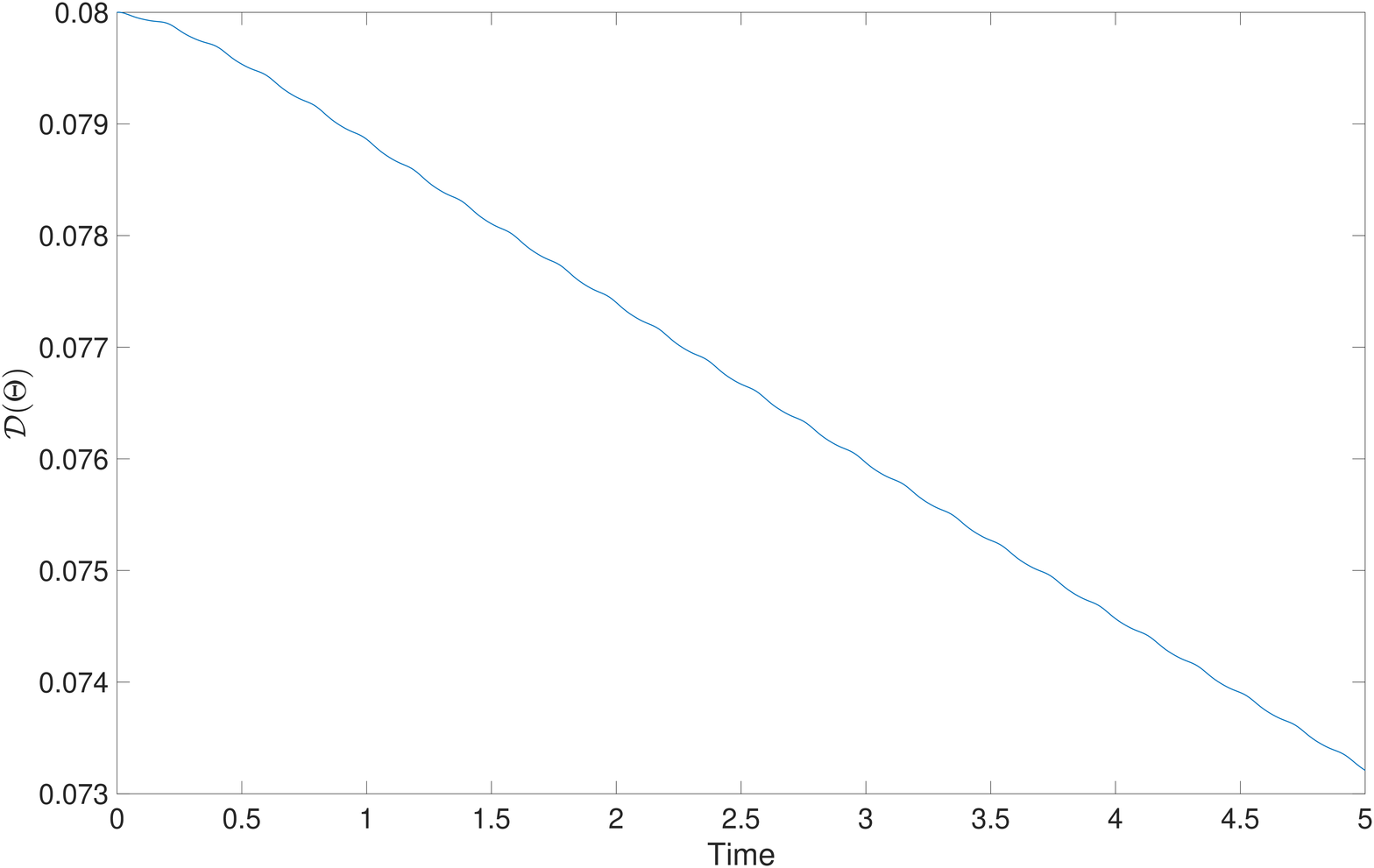}}
			\subfigure[~Graph of $\theta_i(t)/t$ for $0.1\leq t\leq 5$]{\includegraphics[scale = 0.14]{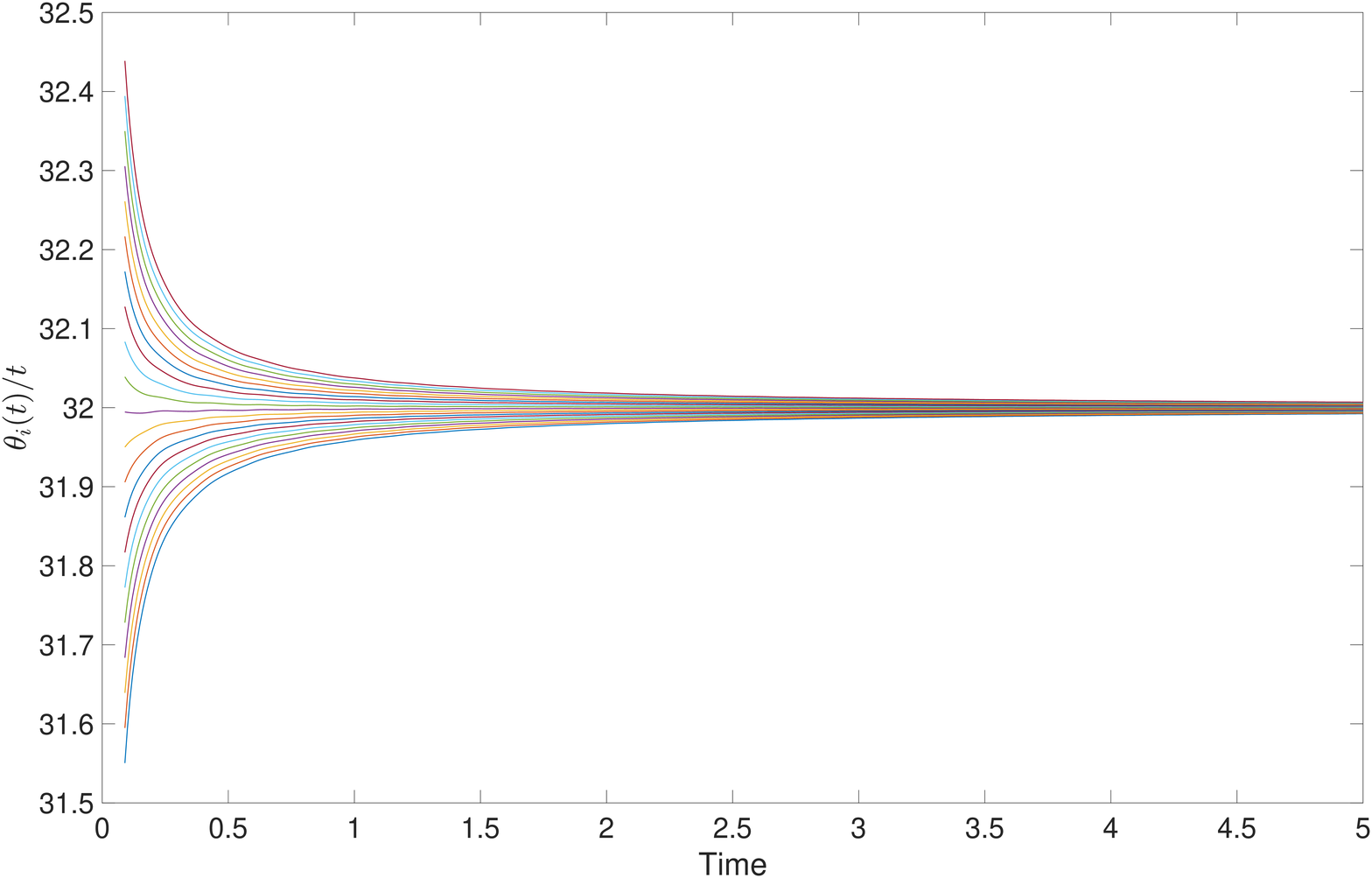}}
		}
		\caption{Emergence of phase-locked state}
		\label{F1}
	\end{figure}
	
		\begin{figure}[h!]
		\centering
		\mbox{ \hspace{-1cm}
			\subfigure[~Graph of $\mathcal D(\Theta(t))$ for $0\leq t\leq 5$ and $\kappa=1$]{\includegraphics[scale = 0.14]{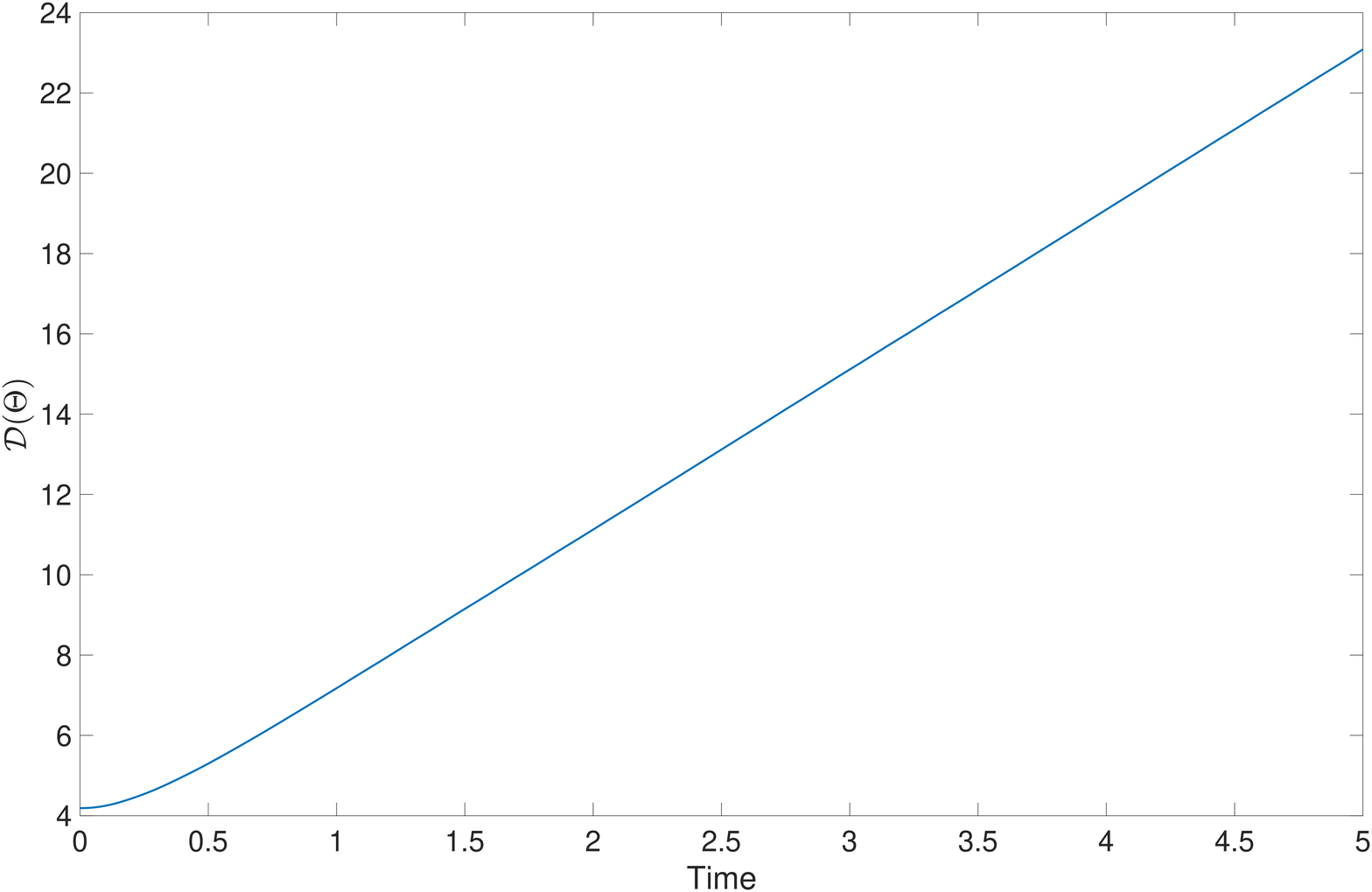}}
			\subfigure[~Graph of $\theta_i(t)/t$ for $0.1\leq t\leq 5$ and $\kappa=1$]{\includegraphics[scale = 0.14]{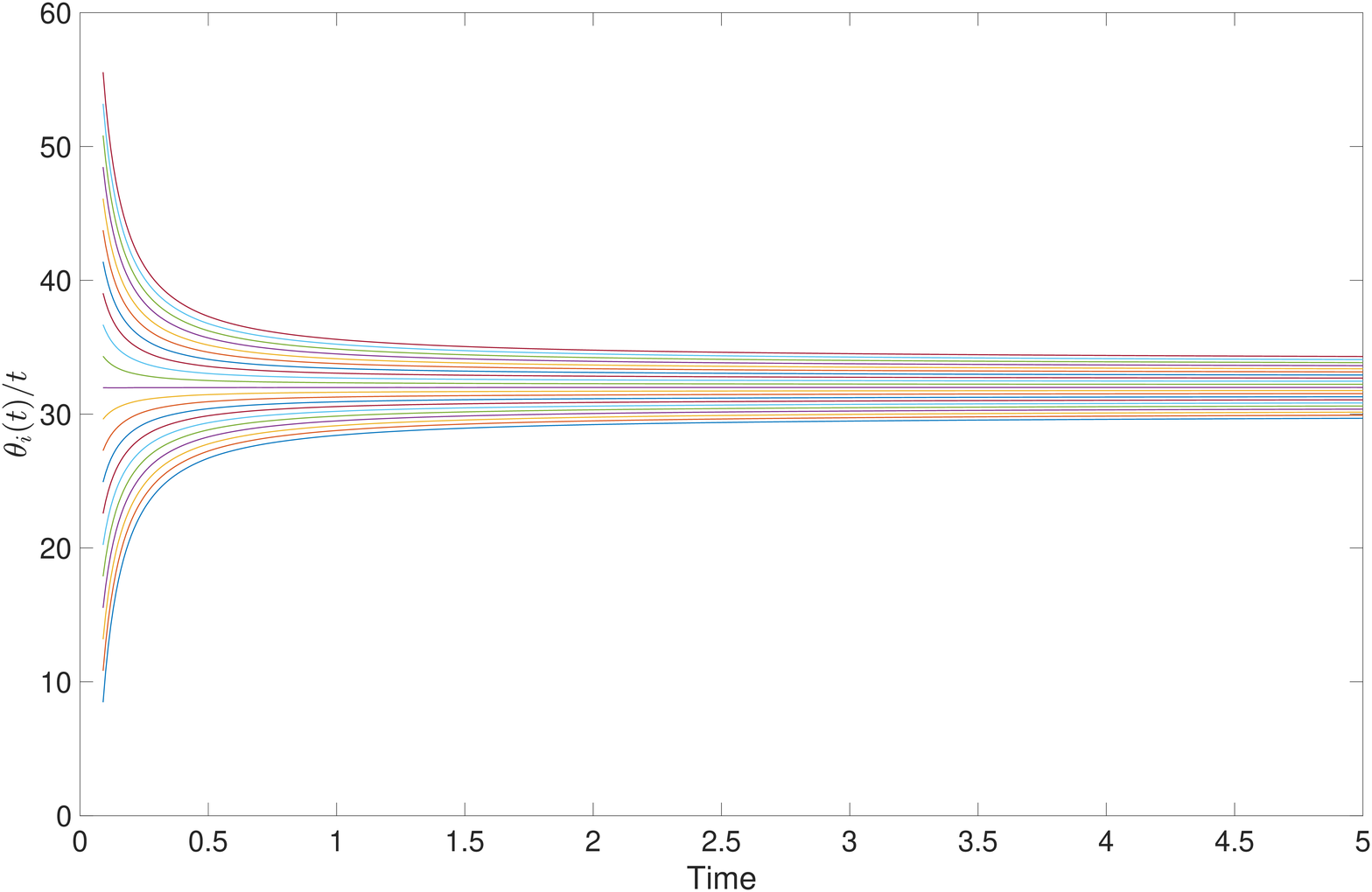}}
		} \\
		\mbox{ \hspace{-1cm}
			\subfigure[~Graph of $\mathcal D(\Theta(t))$ for $0\leq t\leq 5$ and $\kappa=50$]{\includegraphics[scale = 0.14]{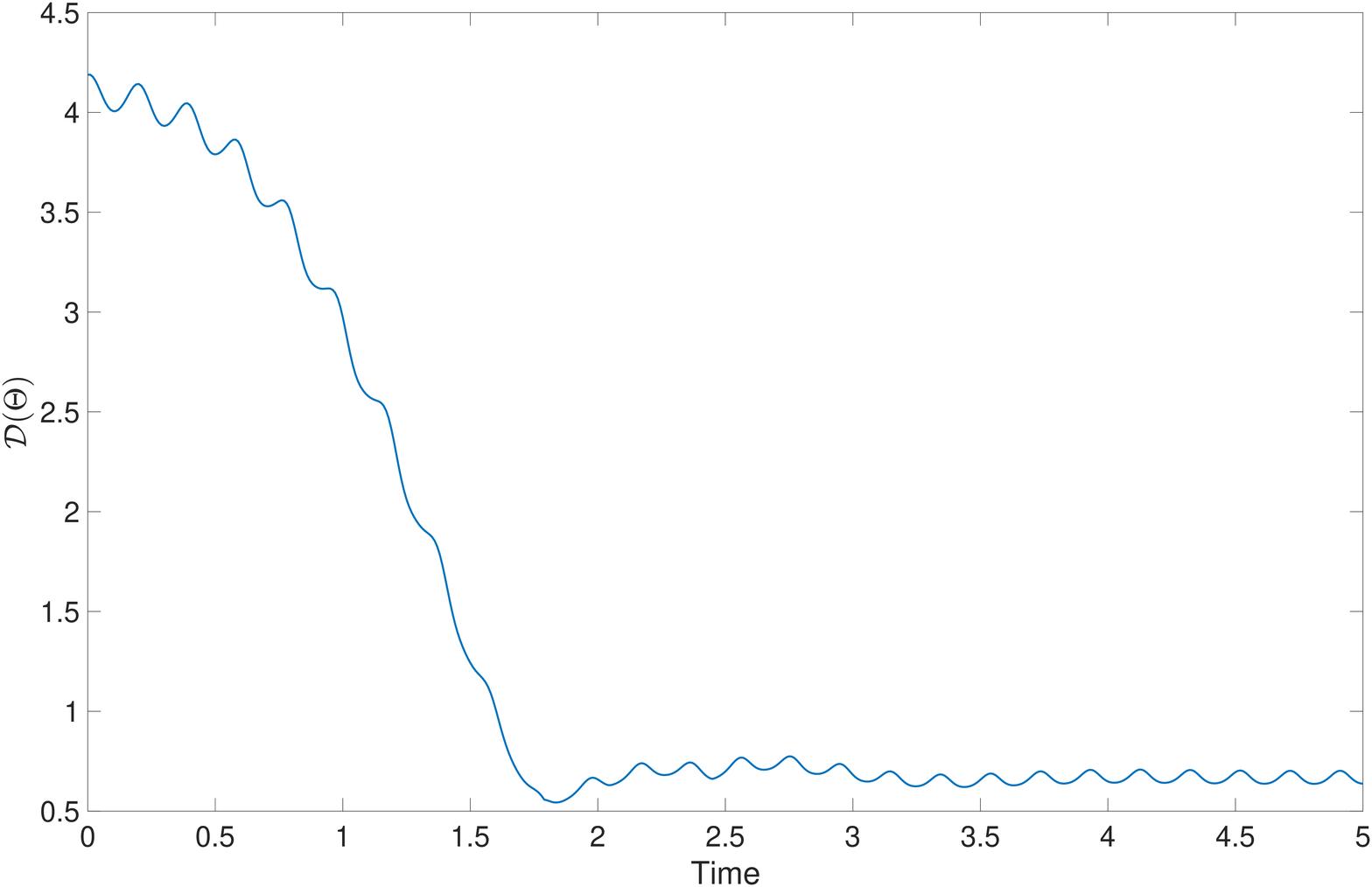}}
			\subfigure[~Graph of $\theta_i(t)/t$ for $0.1\leq t\leq 5$ and $\kappa=50$]{\includegraphics[scale = 0.14]{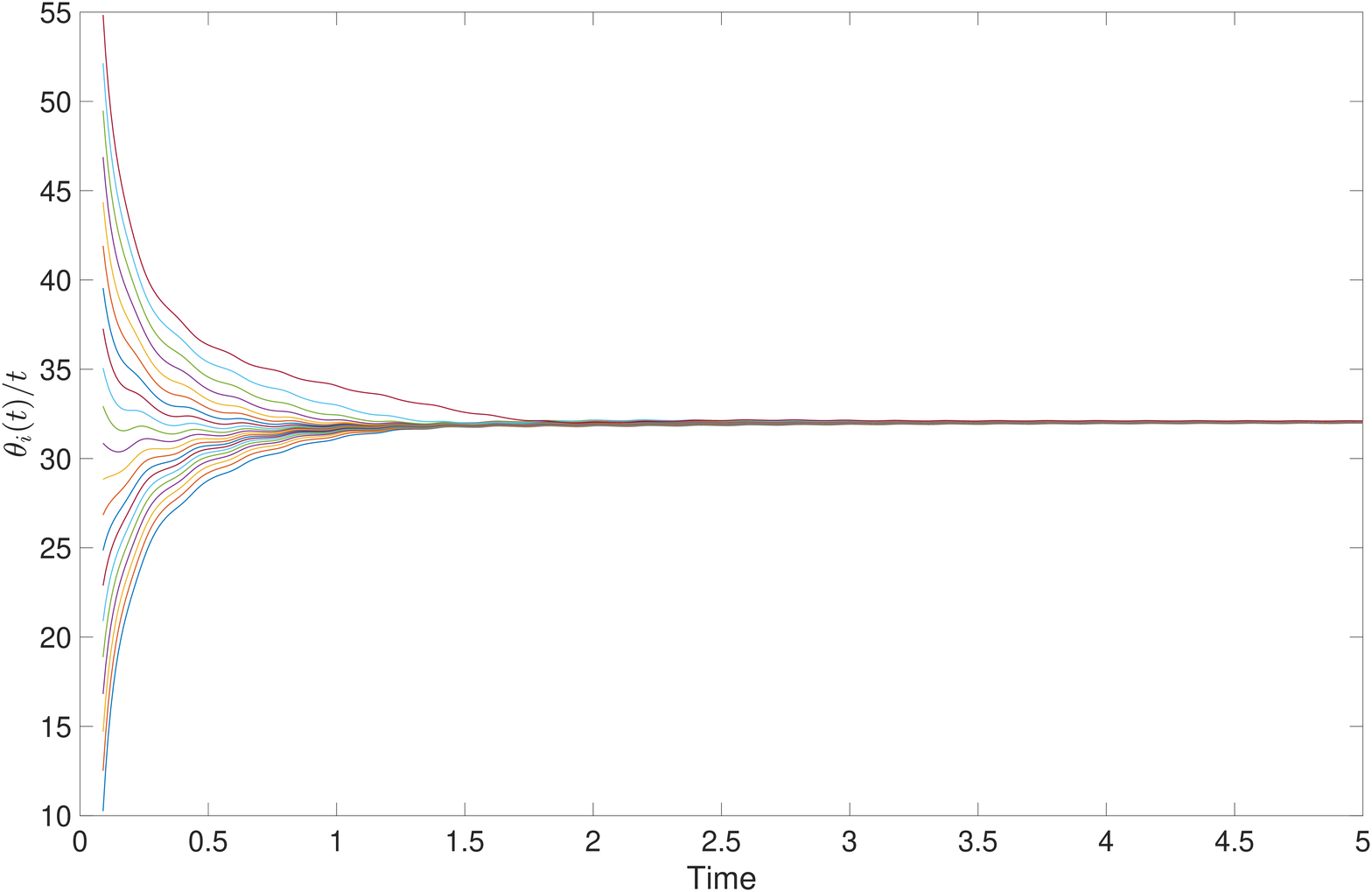}}
		}
		\caption{Effect of coupling strength}
		\label{F2}
	\end{figure}
	
	 In Figure 1, it is seen that the phase diameter is bounded by $\mathcal D(\Theta_0)$ and that the rotation number of each oscillator seems to be well-defined and to coincide with one another, which is in accordance with the assertion of Theorem \ref{thm:det_pl}.
	Next, we change the natural frequencies $\nu_i$ and initial phase $\Theta_0$ to
	\begin{align*}
		\nu_i = 128+8(i-11), \quad \theta^i_0 = \frac{2\pi}{3}(i-11),
	\end{align*}
	and observe the corresponding dynamics for two drastically different coupling strengths, namely $\kappa = 1$ and $\kappa= 50$. In these cases, not all conditions of \eqref{A-5} hold. For $\kappa=1$, synchronous behavior does not seem to emerge, however, for large coupling strength ($\kappa = 50$), the simulations in Figure 2 hint at a phase-locking result in this case as well. More precisely, Figure 2(b) shows the emergence of distinct rotation numbers for $\kappa =1$, which explains the divergence of the phase diameter in Figure 2 (a) and the absence of phase-locking in this case. However, for $\kappa =50$, Figures 2 (c) and (d) hint at the emergence of phase-locking. We infer that phase-locking can occur in suitable large coupling regimes as well.

\subsection{Stochastic case}
We proceed with simulations for the stochastic case, i.e. for the model introduced in \eqref{stoch-Winfree-eq}. First, we choose natural frequencies $\nu_i$, coupling strength $\kappa$, friction coefficient $\gamma$, and initial frequency $\Omega_0$ as
\begin{align*}
& \nu^c = 12, \quad \mathcal D(\nu) = 0, \quad \kappa = 0.1, \quad \gamma = 5, \quad \omega^c_0 = \frac{\nu^c}{\gamma}, \quad \mathcal D(\Omega_0) = 0 \\
& \implies 0 < \nu^c-2\kappa \leq \gamma\omega^c_0 \leq \nu^c+2\kappa,
\end{align*}
and set $D$, $\mathcal{D}(\Theta_0),\delta$ and $\sigma_t$ as
\begin{align*}
& \hspace{-1cm} D = 0.1, \quad \mathcal{D}(\Theta_0) = 0.08,\quad\delta = \frac{\sqrt{\log9}}{50}, \quad \sigma_t = \frac{1}{50(1+t)}, \quad \theta^i_0 = 4\times 10^{-3}(i-11).
\end{align*}Then 
\begin{align*}
&\|\sigma\|_2 = \|\sigma\|_\infty = \frac{1}{50} < \sqrt{\frac{4\kappa}{\gamma}}, \quad \frac{\mathcal D(\Omega_0)}{\cosh \delta}+\gamma \mathcal D(\Theta_0) -2\pi L \alpha_D \approx -0.0744 \leq 0, \\
& 2\pi R \alpha_D+e^\delta \sinh\delta \|Y\mathcal D(\Omega)\|_\infty -\gamma D \approx -0.0094 < 0.
\end{align*}

For this parameter configuration, we observe 5000 sample paths in the time interval $[0, 50]$. 	

\begin{figure}[h!]
\centering
\mbox{ \hspace{-1cm}
\subfigure[~Sample paths of $\mathcal D(\Theta(t))$ for $0\leq t\leq 50$]{\includegraphics[scale = 0.14]{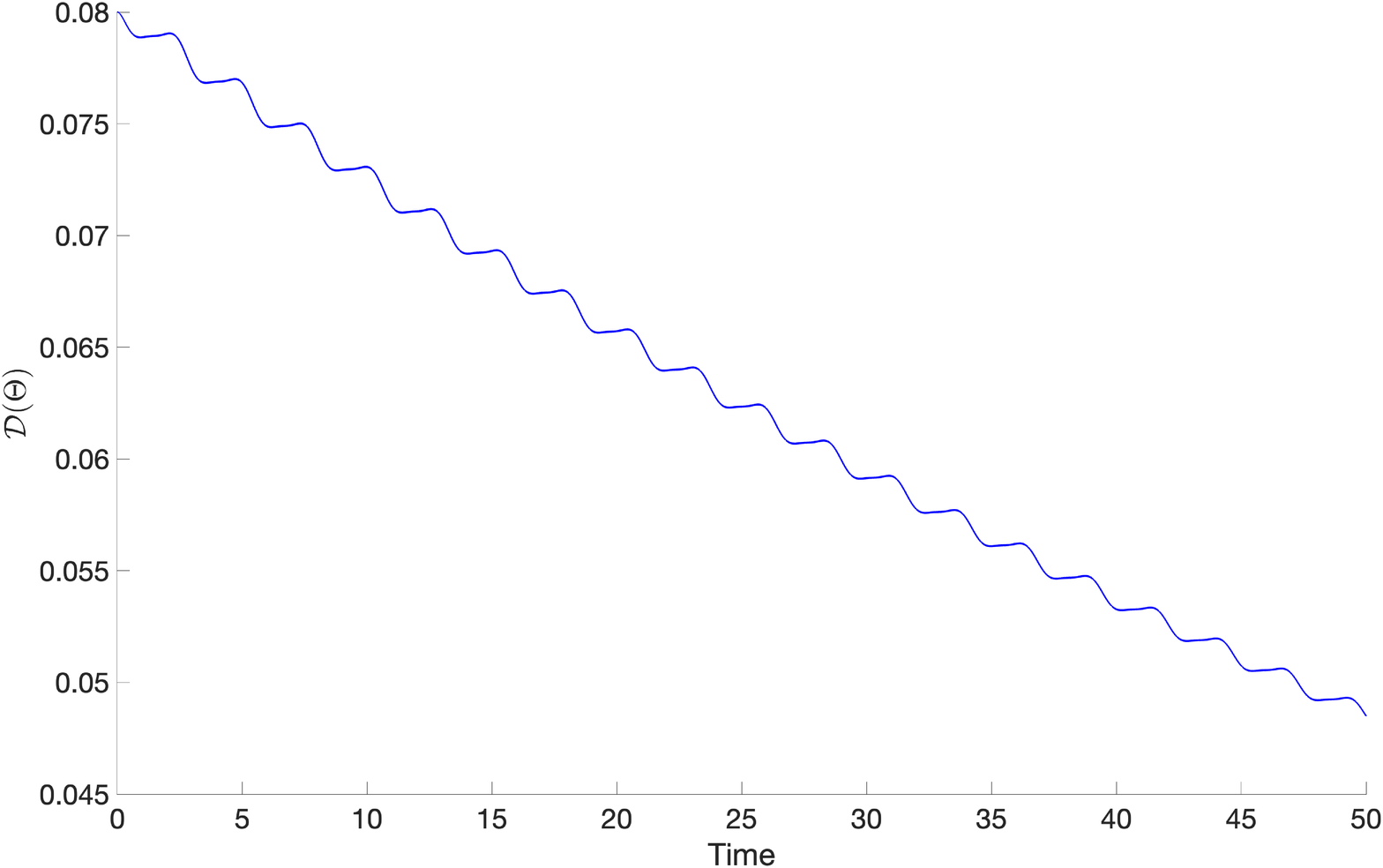}}
\subfigure[~Sample path of $\mathcal D(\Theta(t))$ for $0\leq t\leq 2.5$]{\includegraphics[scale = 0.14]{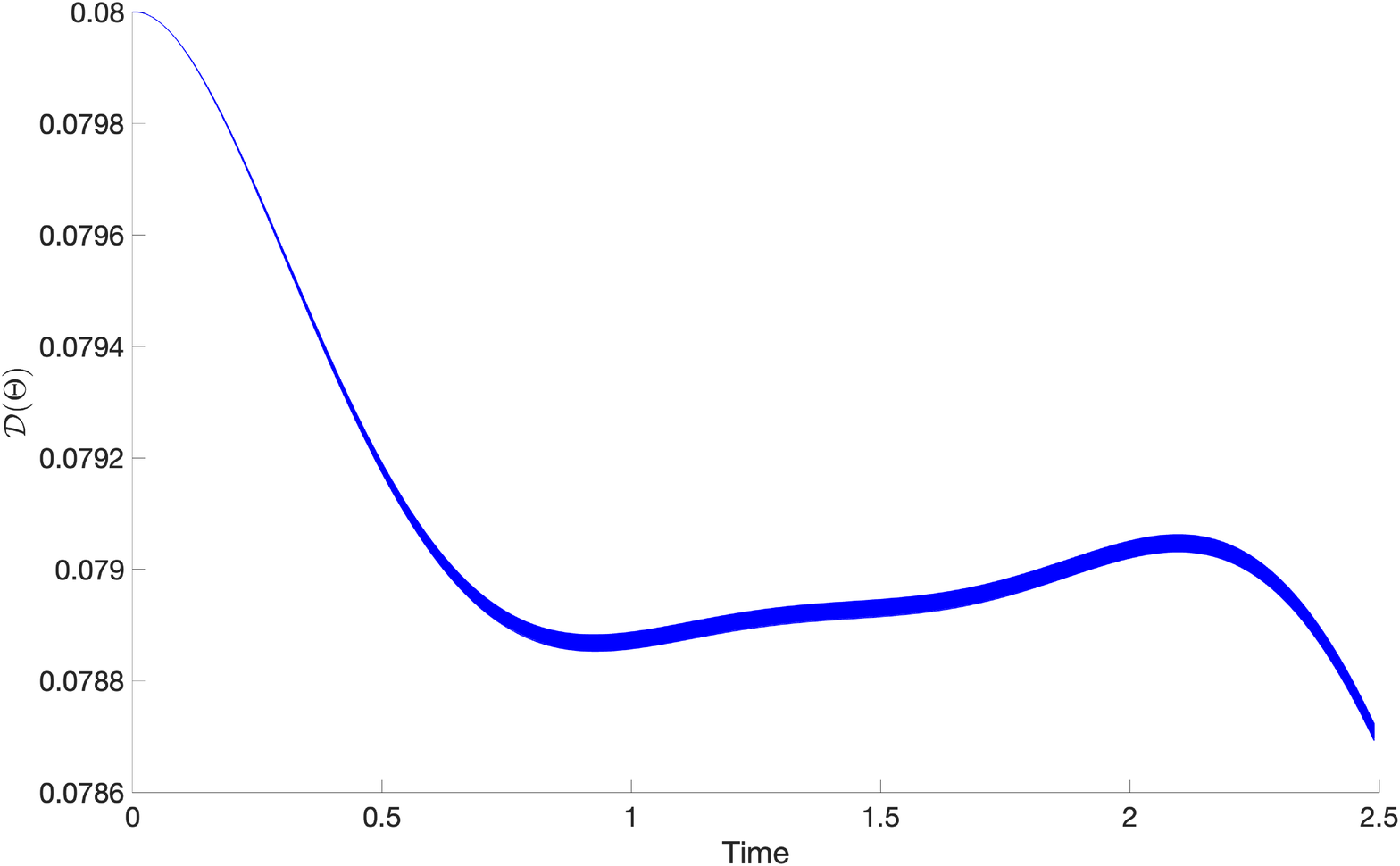}}
}
\caption{Emergence of phase-locked state}
\end{figure}
The corresponding sample paths are plotted in Figure 3, on small and large time scale. One observes that all paths seem to be uniformly (in $t$) bounded by $\mathcal D(\Theta_0)$, which follows the result of Theorem \ref{thm-stoch-case}. Note that the lower bound of the probability for uniformly bounded $\mathcal D(\Theta)$ given in Theorem \ref{thm-stoch-case} is
\begin{align*}
P(A_\delta) \geq 1-2\exp\bigg( -\frac{\delta^2}{2\|\sigma\|_2^2} \bigg) = \frac13.
\end{align*}
Figure Figure 3 suggests that this bound is not optimal. We leave it as a future work.

Next, we change natural frequencies $\nu_i$, initial phase $\Theta_0$, and $\sigma_t$ to
\begin{align*}
\nu_i = 12+\frac{i-11}{10}, \quad \theta^i_0 = \frac{2\pi}{3}(i-11), \quad \sigma_t = \frac{1}{2(1+t)},
\end{align*}
so that $\mathcal D(\nu) = 2$, $\mathcal D(\Theta_0) = 4\pi/3$, and $\|\sigma\|_\infty = \|\sigma\|_2 = 0.5$. We observe 5000 sample paths for coupling strengths $\kappa = 1$ and $\kappa=5$, respectively, in order to separately study the effect of coupling strength on the emergence of phase locking.

\begin{figure}[h!]
\centering
\mbox{ \hspace{-1cm}
\subfigure[~Sample paths of $\mathcal D(\Theta(t))$ for $0\leq t\leq 50$ and $\kappa=1$]{\includegraphics[scale = 0.14]{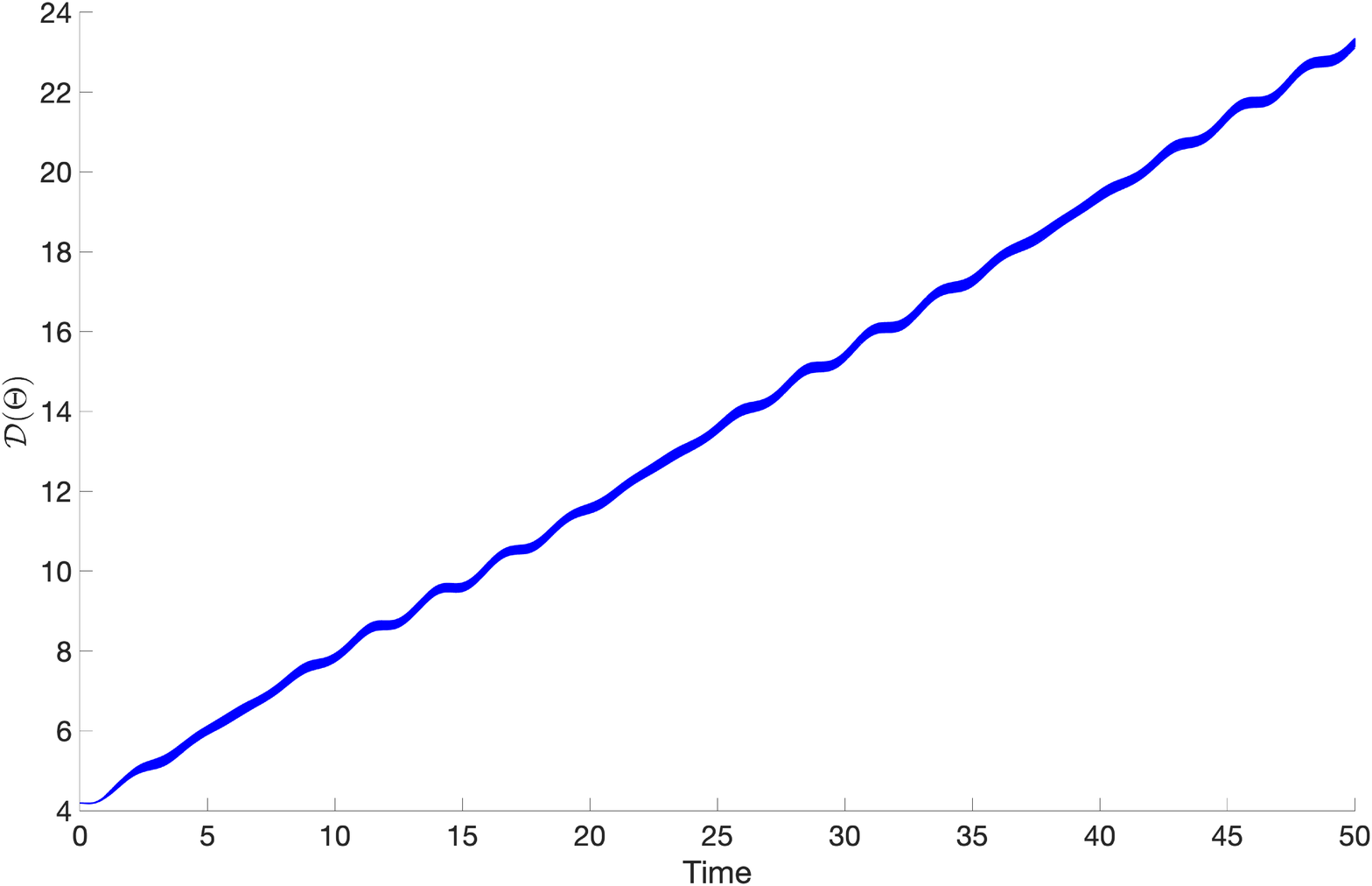}}
\subfigure[~Sample path of $\mathcal D(\Theta(t))$ for $0\leq t\leq 50$ and $\kappa=5$]{\includegraphics[scale = 0.14]{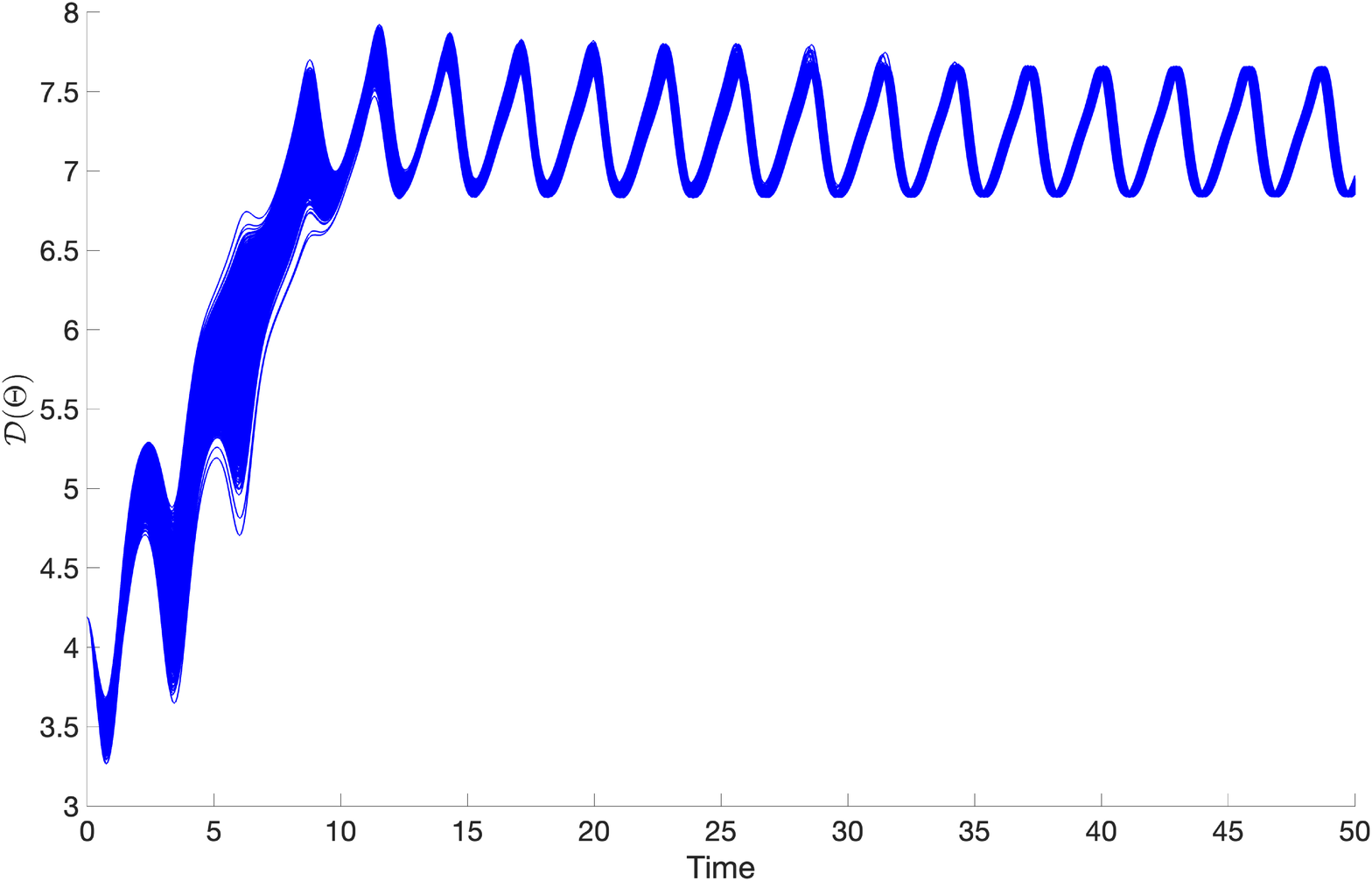}}
}
\caption{Effect of coupling strength}
\end{figure}

Figure 4 (b) suggests that also in a stochastic case, a large coupling regime does not rule out the emergence of phase locking, provided $\kappa$ and $\sigma$ are suitably balanced. We shall investigate the phase transitions in terms of the balance between $\sigma$ and $\kappa$ in the stochastic case more closely in a future work.
\section{Conclusion}\label{sec:conclusion}
We provided sufficient frameworks for phase-locked state emergence for the second-order (stochastic) Winfree model with inertia. In the deterministic case \eqref{sys:WFI}, we obtained uniform in time boundedness of the phase diameter $\max_{i,j}(\theta^i_t-\theta^j_t)$. The key observation towards this result is the following: provided the natural frequencies and the initial data are sufficiently narrowly spread, $\gamma\theta^{ij}+\omega^{ij}$ can be compared to a solution of a differential equation with affine periodic drift.  When the orbit of the periodic part of this drift is suitably small, one can conclude boundedness of the phase diameter. Our numerical simulations suggest that this result can be extended to more general sets of initial data and under milder constraints on the spread of $\nu^i$.

For the stochastic model \eqref{stoch-Winfree-eq}, using a Bernstein-type inequality we obtained lower bounds for the probability of pathwise phase-locking. We note that we did not observe a regularizing effect of the noisy perturbation in terms of the emergence of synchronous behavior, but that we rather had to constrain its effect on the particle system. Indeed, choosing the noise sufficiently small, the lower estimate for the probability of pathwise phase-locking can be made arbitrarily large in $(0,1)$. In future works, it will be interesting to find out whether this is an intrinsic phenomenon of the model or whether refined techniques reveal a certain synchronization by noise effect for the Winfree model with inertia, possibly for other types of multiplicative noise.

	\bibliography{bib-collection}
	\vspace{1cm}
	\noindent \textit{Myeongju Kang} Research Institute of Basic Sciences, Seoul National University, Seoul 08826, Republic of Korea\\
	\textit{E-mail address: bear0117@snu.ac.kr}\\
	\textit{Marco Rehmeier} Faculty of Mathematics, Bielefeld University, Universitätsstraße 25, 33615 Bielefeld, Germany\\
	\textit{E-mail address: mrehmeier@math.uni-bielefeld.de}
\end{document}